\documentclass[11pt]{amsart}

\usepackage{amscd,amsmath,amssymb,amsfonts,amsthm,latexsym}
\usepackage{mathtools}
\usepackage{mathrsfs}

\usepackage[dvipsnames]{xcolor}

\usepackage{enumerate}
\usepackage{enumitem}

\usepackage{aliascnt}

\usepackage[colorlinks=true,
linkcolor=blue,
citecolor=cyan,
urlcolor=NavyBlue]{hyperref}

\usepackage[left=2.5cm,right=2.5cm,
top=3cm,bottom=3cm,
headheight=2.5mm]{geometry}

\allowdisplaybreaks

\newtheorem{theorem}{Theorem}[section]

\newaliascnt{corollary}{theorem}

\aliascntresetthe{corollary}

\newaliascnt{lemma}{theorem}
\newtheorem{lemma}[lemma]{Lemma}
\aliascntresetthe{lemma}

\newaliascnt{proposition}{theorem}
\newtheorem{proposition}[proposition]{Proposition}
\aliascntresetthe{proposition}

\newaliascnt{definition}{theorem}
\newtheorem{definition}[definition]{Definition}
\aliascntresetthe{definition}

\newaliascnt{remark}{theorem}
\newtheorem{remark}[remark]{Remark}
\aliascntresetthe{remark}

\theoremstyle{definition}

\newaliascnt{example}{theorem}

\aliascntresetthe{example}

\DeclareMathOperator{\supp}{supp}

\begin{document}
	
	\title[The Super Alternative Daugavet Property, Unconditional Bases and SCD Geometry]{The Super Alternative Daugavet Property, Unconditional Bases and SCD Geometry}
	
	\author[Ribeiro]{Geivison Ribeiro}
	
	\address[Geivison Ribeiro]
	{
		Departamento de Matem\'{a}tica \newline\indent
		Universidade Federal do Maranh\~{a}o \newline\indent
		65085-580 - S\~{a}o Lu\'is, Brazil.
	}
	
	\email{
		\href{mailto:geivison@unicamp.br}
		{geivison@unicamp.br}
		\textrm{ and }
		\href{mailto:geivison.ribeiro@academico.ufpb.br} {geivison.ribeiro@academico.ufpb.br}}

	\thanks{
		ORCID:
		\href{https://orcid.org/0000-0002-0345-7597}
		{0000-0002-0345-7597}
	}
	
	\author[Rodr\'iguez--Vidanes]{Daniel L. Rodr\'iguez--Vidanes}
	
	\address[Daniel L. Rodr\'iguez--Vidanes]
	{
		Grupo de Análisis Matemático y Aplicaciones (AMA)\newline\indent
		Departamento de Matem\'atica Aplicada a la Ingeniería Industrial\newline\indent
		Escuela Técnica Superior de Ingeniería y Diseño Industrial\newline\indent
		Universidad Politécnica de Madrid\newline\indent
		Ronda de Valencia 3, Madrid, 28012, Spain
	}
	
	\email{
		\href{mailto:dl.rodriguez.vidanes@upm.es}
		{dl.rodriguez.vidanes@upm.es}
	}
	
	\thanks{
		ORCID:
		\href{https://orcid.org/0000-0002-1016-096X}
		{0000-0002-1016-096X}
	}
	
	\keywords{
		Daugavet property,
		super Alternative Daugavet property,
		unconditional basis,
		SCD sets,
		weak topology
	}
	
	\subjclass[2020]{
		46B20,
		46B03,
		46B04,
		46B22,
		46B26
	}
	
\begin{abstract}
	We answer negatively the $1$-unconditional part of \cite[Question~6.4]{Langemets2025}, by Langemets, L\~oo, Martín, Perreau and Rueda Zoca, and, more generally, prove that no infinite-dimensional Banach space with a $K$-unconditional basis, for $1\leq K<3/2$, can satisfy the super Alternative Daugavet property.
	We also address two recent questions posed by L\~oo and Perreau concerning weak topological structures and slicely countably determined phenomena in Banach spaces with unconditional bases. 
	More precisely, we prove that the weak unit ball of every Banach space with a $1$-unconditional basis admits a countable $\pi$-base, thereby giving a positive answer to \cite[Question~5.1]{LooPerreau2026}.
	We further show that every bounded convex subset of a Banach space with a Schauder basis that is shrinking or boundedly complete admits a countable $\pi$-base for its relative weak topology. 
	We complement this result with two permanence principles: one for shrinking Schauder decompositions whose finite partial sums have countable weak $\pi$-bases, and another for unconditional sums over boundedly complete Banach sequence spaces.
	Finally, we prove that, for every $k>1$, there exists a Banach space with a $k$-unconditional basis whose unit ball is not slicely countably determined, thereby giving a positive answer to \cite[Question~5.4]{LooPerreau2026}.
\end{abstract}

\maketitle

\section{Introduction}

The Daugavet property (DP, for short) and its alternative versions play a central role in the geometry of Banach spaces. 
Since the work of Kadets, Shvidkoy, Sirotkin and Werner~\cite{Kadets2000}, rank--one perturbations of the identity have been connected with a rich family of extremal geometric phenomena, slice properties and renorming techniques. 
We refer to \cite{DGZ1993,KadetsBook2025,MartinOikhberg2004,Shvydkoy2000,Becerra2014} for background and to \cite{AbrahamsenAliagaLimaMartinyPerreauProchazkaVeeorg2024,HallerLangemetsPerreauVeeorg2024,Kadets2021,Veeorg2023} for related pointwise and renorming aspects of Daugavet-type geometry.

Recently, Langemets, L\~oo, Mart\'in, Perreau and Rueda Zoca introduced in~\cite{Langemets2025} the \emph{super Alternative Daugavet property} (super ADP, for short).
Recall that a Banach space $X$ has the super ADP if, for every $x\in S_X$ and every relatively weakly open set $W\subset B_X$
with $W\cap S_X\neq\varnothing$, one has
	\[
	\sup_{y\in W} \max_{\theta\in\mathbb T} \|x+\theta y\| = 2.
	\]
This property strengthens the Alternative Daugavet property (ADP, for short) and satisfies
	\[
	\text{DP} \Longrightarrow \text{super ADP} \Longrightarrow \text{ADP},
	\]
with both implications strict; see~\cite{Langemets2025}.
It also imposes strong geometric restrictions, excluding several local regularity properties compatible with the ADP.

A classical obstruction for the DP is its incompatibility with unconditional structures.
In particular, spaces with the DP do not embed into Banach spaces with unconditional bases~\cite{Kadets2000}.
For the super ADP, the corresponding problem was left open in~\cite[Question~6.4]{Langemets2025}: does there exist an infinite-dimensional Banach space with an unconditional basis, or even with a $1$-unconditional basis, and the super ADP?

This question is also connected with recent developments concerning slicely countably determined sets (SCD sets, for short) and weak topological structures in Banach spaces with unconditional bases.
SCD sets were introduced in~\cite{AvilesKadetsMartinMeriShepelska2010} and further studied, among other places, in~\cite{KadetsBook2025,KadetsPerezWerner2018}.
They form a broad class of separable sets including, for instance, separable dentable sets and separable sets not containing sequences equivalent to the canonical basis of $\ell_1$.
The connection with Daugavet-type geometry is natural: the Daugavet property is one of the main known obstructions to slicely countable determination, while SCD assumptions often allow one to strengthen or identify alternative Daugavet-type properties.


Two classical questions in this direction asked whether every separable non-SCD Banach space must contain an isomorphic copy of a Banach space with the DP, and whether every Banach space with an unconditional basis must be SCD; see~\cite{AvilesKadetsMartinMeriShepelska2010,KadetsPerezWerner2018}. 
Since spaces with the DP do not embed into spaces with unconditional bases, a negative answer to the second question would also provide a negative answer to the first one. 
L\~oo and Perreau recently obtained such a negative answer in~\cite{LooPerreau2026} by means of the binary tree space $X_T$. This space is generated by the adequate family of all chains in the infinite binary tree, in the sense of Talagrand~\cite{Talagrand1979,Talagrand1984}, and has a canonical $1$-unconditional basis. 
They proved that its positive unit ball $B_{X_T}^{+}$ is not SCD~\cite[Theorem~3.6]{LooPerreau2026}, whereas the relative weak topology of the whole unit ball $B_{X_T}$ admits a countable $\pi$-base~\cite[Theorem~1.6]{LooPerreau2026}. 
In particular, $B_{X_T}$ is SCD, since every bounded convex set admitting a countable $\pi$-base for its relative weak topology is SCD~\cite[Proposition~2.21]{AvilesKadetsMartinMeriShepelska2010}, while $B_{X_T}^{+}$ cannot admit such a $\pi$-base. 
Thus, even in the presence of a $1$-unconditional basis, the weak-topological and SCD behaviour of the whole unit ball may differ sharply from that of its natural bounded closed convex subsets. 
This should also be compared with the result of Kadets, Martín, Merí and Werner~\cite{KadetsMartinMeriWerner2013}, according to which $B_X$ is SCD whenever $X$ has a $1$-unconditional basis.

Countable $\pi$-bases for weak topologies provide another important tool in this setting.
If a bounded convex set admits a countable $\pi$-base for its relative weak topology, then it is SCD.
Conversely, for separable convex bounded sets, the failure of a countable $\pi$-base forces the presence of a sequence equivalent to the canonical basis of $\ell_1$; see~\cite{AvilesKadetsMartinMeriShepelska2010}.
L\~oo and Perreau emphasize that there are few general mechanisms for proving that a weak topology has no countable $\pi$-base.
One such mechanism is precisely the super ADP: they prove in~\cite{LooPerreau2026} that if $X$ has the super ADP, then $(B_X,w)$ has no countable $\pi$-base.
In particular, every Banach space with the super ADP contains an isomorphic copy of $\ell_1$, answering a question left open in~\cite{Langemets2025}.

This observation gives a direct link between the super ADP and the question whether weak unit balls of spaces with $1$-unconditional bases must admit countable $\pi$-bases.
Indeed, if there existed an infinite-dimensional Banach space with a $1$-unconditional basis and the super ADP, then the result of L\~oo and Perreau would immediately provide a negative answer to~\cite[Question~5.1]{LooPerreau2026}.
Equivalently, it would give a natural solution to the approach suggested in~\cite[Question~5.2]{LooPerreau2026}, where the possible coexistence of the super ADP with SCD unit balls, and in particular with $1$-unconditional bases, is singled out as a central open problem.
L\~oo and Perreau prove that this coexistence fails for Banach spaces generated by adequate families, including the tree spaces considered in their paper.
Our main result strengthens this conclusion by showing that no infinite-dimensional Banach space with a $K$-unconditional basis, for $1\leq K<3/2$, can have the super ADP.

\begin{theorem}\label{thm:intro-main}
	Let $K\in [1,3/2)$.
	Then, there is no infinite-dimensional Banach space with a $K$-unconditional basis having the super ADP.
\end{theorem}

We show that if a Banach space with a $K$-unconditional basis has the super ADP, then, for every $n\in\mathbb N$ and every $\eta>0$, it contains successive normalized block vectors that are $K(1+\eta)$-equivalent to the canonical basis of $\ell_1^n$. 
More precisely, these vectors can be chosen so that their positive sum has norm arbitrarily close to $n$ and they satisfy the canonical upper $\ell_1$-estimate. 
A finite-dimensional obstruction based on balanced sign combinations then shows that such block structures are incompatible with the local extremal behaviour required by the super ADP whenever $K<3/2$.

The approach is related to the finite representability programme of Krivine and Rosenthal in~\cite{Rosenthal1978} and Odell and Schlumprecht in~\cite{OdellSchlumprecht1993}, although the $\ell_1$-geometry is obtained here directly from the super ADP through a block-growth mechanism.
This reduction allows us to work entirely within the $\ell_1$-framework.


Having established the stronger obstruction to the super ADP for $K$-unconditional bases with $1\leq K<3/2$, we next turn to the weak-topological question posed in~\cite[Question~5.1]{LooPerreau2026}. 
L\~oo and Perreau proved that $(B_X,w)$ has a countable $\pi$-base for the binary tree space, the modified binary tree space and the countably branching tree space. 
We give a complete positive answer to their question by proving that $(B_X,w)$ admits a countable $\pi$-base whenever $X$ has a $1$-unconditional basis. 
Since the argument relies essentially on the invariance of the unit ball under coordinatewise unimodular multipliers, it does not extend directly to arbitrary bounded convex subsets. 
We therefore also show that every bounded convex subset of $X$ admits a countable $\pi$-base for its relative weak topology whenever $X$ has a Schauder basis that is shrinking or boundedly complete, and we further develop two permanence principles extending these mechanisms to Schauder decompositions and unconditional sums.

Finally, we address a quantitative renorming problem from~\cite{LooPerreau2026}.
Starting from the binary tree space $X_T$, L\~oo and Perreau show that a suitable symmetrization of its positive unit ball gives an equivalent renorming with a $2$-unconditional basis whose unit ball is not SCD.
They ask whether the constant $2$ can be improved, that is, whether one can obtain a Banach space with a $k$-unconditional basis, for some $k\in(1,2)$, whose unit ball is not SCD; see~\cite[Question~5.4]{LooPerreau2026}.
We answer this question in the strongest possible form: for every $k>1$, there exists a Banach space with a $k$-unconditional basis whose unit ball is not SCD.
This result does not address the stronger problem of whether one can find an unconditional basis whose unit ball has no SCD points, asked in~\cite[Question~5.5]{LooPerreau2026}, but it shows that non-SCD unit balls occur arbitrarily close to the $1$-unconditional regime.


Taken together, these results reveal a sharp quantitative contrast. 
On the one hand, the super ADP is incompatible with every $K$-unconditional basis whenever $1\leq K<3/2$. 
On the other hand, exact $1$-unconditionality guarantees that the weak unit ball admits a countable $\pi$-base, whereas for every $k>1$ there exists a Banach space with a $k$-unconditional basis whose unit ball is not SCD. 
At the same time, the binary tree space shows that exact $1$-unconditionality alone does not ensure the existence of countable weak $\pi$-bases for arbitrary bounded convex subsets; this stronger conclusion is recovered under additional hypotheses such as shrinkingness or bounded completeness. 
Thus, the paper connects three phenomena treated in~\cite{LooPerreau2026}: the weak-topological obstruction produced by the super ADP, the existence of countable $\pi$-bases for weak unit balls, and the quantitative instability of SCD behaviour under renormings with unconditional constants arbitrarily close to $1$.


The paper is organized as follows. 
In Section~\ref{sec:preliminares}, we collect the basic notation concerning bases, block sequences and finite equivalence, and recall the notions of countable $\pi$-bases and SCD sets. 
In Section~\ref{sec:l1-reduction}, we establish the asymptotic $\ell_1$-structure forced by the super ADP in spaces with unconditional bases. 
In Section~\ref{sec:finite-obstructions}, we develop a finite-dimensional obstruction and use it to prove Theorem~\ref{thm:intro-main}. 
In Section~\ref{sec:pibase}, we prove that the weak unit ball of every Banach space with a $1$-unconditional basis admits a countable $\pi$-base, thereby answering~\cite[Question~5.1]{LooPerreau2026}. 
We also show that every bounded convex subset of a Banach space with a Schauder basis that is shrinking or boundedly complete admits a countable $\pi$-base for its relative weak topology, and establish two permanence principles extending these results to shrinking Schauder decompositions and unconditional sums over boundedly complete Banach sequence spaces. 
Finally, in Section~\ref{sec:non-scd}, we prove that, for every $k>1$, there exists a Banach space with a $k$-unconditional basis whose unit ball is not SCD.

\section{Preliminaries}\label{sec:preliminares}

In this section, we collect the notation and auxiliary material used throughout the paper.
All Banach spaces are considered over the scalar field $\mathbb K\in\{\mathbb R,\mathbb C\}$, unless stated otherwise, with $\mathbb R$ being the real numbers and $\mathbb C$ the complex numbers.
We denote $\mathbb T:=\{z\in\mathbb C \colon |z|=1\}$ the unit circle.
In the real case, $\mathbb T := \{-1,1\}$.
If $z\in \mathbb C$, we denote the real part of $z$ by $\operatorname{Re}(z)$.
Given a Banach space $X$, we also denote by $S_X$ and $B_X$ the unit sphere and closed unit ball of $X$, respectively, and $X^*$ the (topological) dual of $X$.
Also, if $A\subset X$, the diameter of $A$ is denoted by $\operatorname{diam}(A)$.

If $X$ is a Banach space with a Schauder basis $(e_j)_{j=1}^\infty$, we will denote $x\in c_{00}$ for a vector $x\in X$ if there exists $j_0\in \mathbb N$ such that $x$ can be written as
	$$
	x=\sum_{j=1}^{j_0} a_j e_j.
	$$
Let us now recall the notion of a $k$-unconditional basis for some $k\geq 1$.

\begin{definition}
	Let $X$ be a Banach space with a Schauder basis $(e_j)_{j=1}^\infty$ and let $k\geq1$.
	We say that $(e_j)_{j=1}^\infty$ is \emph{$k$-unconditional} if, for every $N\in\mathbb N$ and every choice of scalars $a_1,\dots,a_N$, $b_1,\dots,b_N$ satisfying $|a_j|\le |b_j|$ for every $j \in \{1,\dots,N\}$, we have
		\[
		\left\|\sum_{j=1}^N a_je_j\right\| \le k\left\|\sum_{j=1}^N b_je_j\right\|.
		\]
\end{definition}

For every finite set $A\subset\mathbb N$, we denote by
	\[
	P_Ax := \sum_{j\in A} e_j^*(x)e_j,
	\]
for all $x\in X$, the corresponding coordinate projection.
If the basis is $1$-unconditional, then $\|P_A\|\le1$ for every finite set $A\subset\mathbb N$ and, for all $\varepsilon =(\varepsilon_j)_{j=1}^\infty \subset \mathbb T$, the coordinate multiplier operator
	$$
	S_\varepsilon \left( \sum_{j=1}^\infty a_j e_j \right) = \sum_{j=1}^\infty \varepsilon_j a_j e_j
	$$
is an isometry.	

For $N\in\mathbb N$, we also write
	\[
	P_Nx := \sum_{j=1}^N e_j^*(x)e_j,
	\]
for all $x\in X$, the canonical basis projection onto $\operatorname{span}\{e_1,\dots,e_N\}$.

\medskip

We shall repeatedly work with block vectors and block sequences.
We recall below the necessary definitions and notation.

\begin{definition}
	Let $(e_j)_{j=1}^\infty$ be a Schauder basis of $X$.
	A sequence $(u_k)_{k\in I}$, where either $I=\{1,\dots,n\}$ for some $n\in \mathbb N$ or $I=\mathbb N$, is called a \emph{block sequence} of $(e_j)_{j=1}^\infty$ if there exist finite sets $F_k\subset\mathbb N$ such that $\max F_k<\min F_{k+1}$ and $u_k\in\operatorname{span}\{e_j \colon j\in F_k\}$ for every $k\in I$.
	If, moreover, $\|u_k\|=1$ for every $k\in I$, we say that $(u_k)_{k\in I}$ is a \emph{normalized block sequence}.
\end{definition}

The support of a vector
	\[
	x=\sum_{j=1}^\infty a_je_j \in X
	\]
with respect to the basis $(e_j)_{j=1}^\infty$ is denoted by
	\[
	\supp(x) := \{j\in\mathbb N \colon a_j\neq0\}.
	\]

\medskip

We next recall the standard notion of finite equivalence between sequences.

\begin{definition}
	Let $n\in \mathbb N$, $(x_j)_{j=1}^n\subset X$ and $(y_j)_{j=1}^n\subset Y$ be finite sequences of nonzero vectors, and let $K\ge1$.
	We say that $(x_j)_{j=1}^n$ and $(y_j)_{j=1}^n$ are \emph{$K$-equivalent} if, for every choice of scalars $a_1,\dots,a_n$, we have
		\[
		\frac1K \left\|\sum_{j=1}^n a_jy_j\right\| \le \left\|\sum_{j=1}^n a_jx_j\right\| \le K\left\|\sum_{j=1}^n a_jy_j\right\|.
		\]
\end{definition}

In particular, a finite sequence $(u_1,\dots,u_n)\subset X$ is $K$-equivalent to the canonical basis of $\ell_1^n$ if
	\[
	\frac1K\sum_{j=1}^n |a_j| \le \left\|\sum_{j=1}^n a_ju_j\right\| \le K\sum_{j=1}^n |a_j|
	\]
for every choice of scalars $a_1,\dots,a_n$.

\medskip

We shall also use the classical notion of a shrinking basis and a boundedly complete basis.

\begin{definition}
	Let $X$ be a Banach space with a Schauder basis $(e_j)_{j=1}^\infty$.
	We say that $(e_j)_{j=1}^\infty$ is \emph{shrinking} if the associated biorthogonal sequence $(e_j^*)_{j=1}^\infty$ is a Schauder basis of $X^*$.
	We say that $(e_j)_{j=1}^\infty$ is \emph{boundedly complete} if $\sum_{j=1}^\infty a_j e_j$ converges whenever the sequence of scalars $(a_j)_{j=1}^\infty$ satisfies
		$$
		\sup_{n\in \mathbb N} \left\|\sum_{j=1}^n a_j e_j\right\| < \infty.
		$$
\end{definition}

We next recall the notion of a countable $\pi$-base.

\begin{definition}
	Let $(T,\tau)$ be a topological space.
	A family $\mathcal U$ of nonempty open subsets of $T$ is called a \emph{$\pi$-base} for $(T,\tau)$ if every nonempty open subset of $T$ contains some member of $\mathcal U$.
	If $\mathcal U$ can be chosen countable, we say that $(T,\tau)$ has a \emph{countable $\pi$-base}.
\end{definition}

Finally, we recall the notions related to SCD sets.
We follow the standard terminology introduced in~\cite{AvilesKadetsMartinMeriShepelska2010} (see also~\cite{KadetsPerezWerner2018,KadetsBook2025}).
We shall also use the pointwise version of slicely countable determination introduced in~\cite{LangemetsLooMartinRueda2024}.

\begin{definition}
	Let $A$ be a bounded subset of a Banach space $X$.
	A \emph{slice} of $A$ is a nonempty subset of the form
		\[
		S(A,x^*,\alpha) := \left\{ x\in A \colon \operatorname{Re}x^*(x)>\sup_{a\in A}\operatorname{Re}x^*(a)-\alpha\right\},
		\]
	where $x^*\in X^*\setminus{0}$ and $\alpha>0$.
\end{definition}

\begin{definition}
	Let $A$ be a bounded subset of a Banach space $X$, and let $x\in A$.
	A sequence of slices $(S_n)_{n=1}^\infty$ of $A$ is \emph{determining for $x$} if, whenever one chooses $x_n\in S_n$ for every $n\in\mathbb N$, one has $x\in \overline{\operatorname{conv}} \{x_n \colon n\in\mathbb N\}$.
\end{definition}

\begin{definition}
	Let $A$ be a bounded subset of a Banach space $X$, and let $x\in A$.
	We say that $x$ is a \emph{slicely countably determined point} of $A$, or an \emph{SCD point} of $A$, if there exists a sequence of slices of $A$ which is determining for $x$.
\end{definition}

\begin{definition}
	Let $A$ be a bounded subset of a Banach space $X$.
	We say that $A$ is \emph{slicely countably determined} (SCD) if there exists a sequence of slices $(S_n)_{n=1}^\infty$ of $A$ such that, whenever $B\subset A$ intersects every $S_n$, we have $A\subset \overline{\operatorname{conv}}(B)$.
	(For bounded closed convex separable sets, this is equivalent to requiring that every point of $A$ is an SCD point.)
\end{definition}

\section{Asymptotic $\ell_1$ structure and the super ADP}\label{sec:l1-reduction}

We will show that the super ADP forces the presence of asymptotic $\ell_1$-type structures.
This allows us to avoid the general finite-representability reduction and to work directly in the $\ell_1$-framework.


\medskip

We begin by establishing a linear growth mechanism for sums of blocks.


\begin{lemma}\label{lem:super-adp-linear-growth}
	Let \(X\) be a Banach space with a Schauder basis $(e_j)_{j=1}^\infty$. 
	If \(X\) has the super ADP, then for every $m\in\mathbb N$ and $\varepsilon>0$, there exist successive blocks $v_1,\dots,v_m\in S_X$ and scalars $\theta_1,\dots,\theta_m\in\mathbb T$ such that
	\[
	\left\|\sum_{j=1}^m \theta_j v_j\right\| > (1-\varepsilon)m.
	\]
\end{lemma}

\begin{proof}
	Fix \(m\in\mathbb N\) and \(\varepsilon>0\). If \(m=1\), simply take any block \(v_1\in S_X\) and define \(\theta_1=1\). Then
		\[
		\left\|\sum_{j=1}^m\theta_jv_j\right\| = 1 > (1-\varepsilon)m,
		\]
	and the result follows.
	
	Assume now that \(m>1\). 
	Choose \(\gamma\in(0,1/3)\) sufficiently small so that
		\[
		\gamma(m^2+6m)<\min\{\varepsilon m,1/2\}.
		\]
	We will construct inductively successive blocks $v_1,\dots,v_m\in S_X$ and scalars $\theta_1,\dots,\theta_m\in\mathbb T$ satisfying
		\begin{equation}\label{eq:inductive-growth}
			\left\| \sum_{j=1}^k\theta_jv_j\right\| \ge k-\gamma\sum_{j=1}^{k-1}(j+6),
		\end{equation}
	for every $k\in \{1,\ldots,m\}$, where the empty sum is understood to be zero when $k=1$.
	For convenience, write
		\[
		S_k := \sum_{j=1}^k\theta_jv_j,
		\]
	for every $k\in \{1,\ldots,m\}$.
	
	Choose first any block \(v_1\in S_X\) and define \(\theta_1=1\). 
	Then \(S_1=v_1\), and therefore
		\[
		\|S_1\| = 1 = 1-\gamma\sum_{j=1}^{0}(j+6).
		\]
	Hence, \eqref{eq:inductive-growth} holds for \(k=1\).
	
	Now assume that for some \(1\le k<m\), we already have blocks $v_1,\dots,v_k\in S_X$ and scalars $\theta_1,\dots,\theta_k \in \mathbb T$ satisfying \eqref{eq:inductive-growth}. 
	Since
		\[
		\sum_{j=1}^{k-1}(j+6) < \sum_{j=1}^{m-1}(j+6) = \frac{m(m-1)}2+6(m-1) < m^2+6m,
		\]
	the choice of \(\gamma\) yields $\|S_k\| > k-1/2 > 0$.
	In particular, \(S_k\neq0\), so we may define
		\begin{equation}\label{eq:def-xk}
			x_k:=\frac{S_k}{\|S_k\|}\in S_X.
		\end{equation}
	
	Choose \(N_k\in\mathbb N\) sufficiently large so that
		\[
		\bigcup_{j=1}^k \supp(v_j) \subset \{1,\dots,N_k\},
		\]
	and consider the set
	\[
	W_k:=
	\{z\in B_X \colon \|P_{N_k}z\|<\gamma\}.
	\]
	
	Observe that \(W_k\cap S_X\neq\varnothing\). Indeed, since the basis is
	infinite, the vector $\frac{e_{N_k+1}}{\|e_{N_k+1}\|} \in S_X$ satisfies \(P_{N_k} \frac{e_{N_k+1}}{\|e_{N_k+1}\|} = 0\). 
	Hence, \(\frac{e_{N_k+1}}{\|e_{N_k+1}\|} \in W_k\cap S_X\).
	Moreover, \(W_k\) is relatively weakly open in \(B_X\) since \(P_{N_k}\) has finite rank. 
	Therefore, by the super ADP applied to \(x_k\) and \(W_k\), there exist $z_{k+1}\in W_k$ and $\theta_{k+1}\in\mathbb T$ such that
		\begin{equation}\label{eq:almost-daugavet}
			\|x_k+\theta_{k+1}z_{k+1}\|>2-\gamma.
		\end{equation}
	In particular, $\|z_{k+1}\|>1-\gamma$.
	
	Define $y_{k+1}:=(I-P_{N_k})z_{k+1}$.
	Since \(z_{k+1}\in W_k\), we have
	\begin{equation}\label{eq:small-head}
		\|P_{N_k}z_{k+1}\|<\gamma.
	\end{equation}
	Consequently,
	\begin{equation}\label{eq:large-tail}
		\|y_{k+1}\|
		\ge
		\|z_{k+1}\|-\|P_{N_k}z_{k+1}\|
		>
		1-2\gamma.
	\end{equation}
	As \((e_j)_{j=1}^\infty\) is a Schauder basis, the sequence \((P_Mz_{k+1})_{M=1}^\infty\) converges to \(z_{k+1}\) in norm. 
	Hence,
		\[
		(P_M-P_{N_k})z_{k+1} = P_Mz_{k+1}-P_{N_k}z_{k+1} \longrightarrow z_{k+1}-P_{N_k}z_{k+1} = y_{k+1},
		\]
	as $M\to \infty$.
	Therefore, we may choose \(M_k>N_k\) such that, defining $q_{k+1}:=(P_{M_k}-P_{N_k})z_{k+1}$, we have
		\begin{equation}\label{eq:q-close}
			\|q_{k+1}-y_{k+1}\|<\gamma.
		\end{equation}
	Combining \eqref{eq:large-tail} and \eqref{eq:q-close}, we obtain
		\begin{equation}\label{eq:q-large}
			\|q_{k+1}\| \ge \|y_{k+1}\|-\|q_{k+1}-y_{k+1}\| > 1-3\gamma > 0.
		\end{equation}
	
	Define now
		\[
		v_{k+1}:=\frac{q_{k+1}}{\|q_{k+1}\|} \in S_X.
		\]
	Note that \(v_{k+1}\) is supported after \(N_k\). 
	In particular, $v_1,\dots,v_k,v_{k+1}$ remain successive normalized blocks.
	We claim that
		\begin{equation}\label{eq:v-close}
			\|v_{k+1}-z_{k+1}\|<5\gamma.
		\end{equation}
	Indeed, first by \eqref{eq:small-head} and \eqref{eq:q-close} we have
		\begin{align}
			\|v_{k+1}-z_{k+1}\| & \le \|v_{k+1}-q_{k+1}\| + \|q_{k+1}-y_{k+1}\| + \|y_{k+1}-z_{k+1}\| \nonumber \\
			& = \|v_{k+1}-q_{k+1}\| + \|q_{k+1}-y_{k+1}\| + \|P_{N_k} z_{k+1}\| \nonumber \\
			& < \|v_{k+1}-q_{k+1}\| + 2\gamma. \label{eq:triangle}
		\end{align}
	Also, by \eqref{eq:almost-daugavet} and the fact that $|\theta_{k+1}|=\|x_k\|=1$, we obtain $\|z_{k+1}\|> 1-\gamma$.
	Therefore, since $\|z_{k+1}\|\leq 1$, it follows that $0\leq 1-\|z_{k+1}\|< \gamma$.
	Consequently, again by \eqref{eq:small-head} and \eqref{eq:q-close} we have
		$$
		\|v_{k+1}-q_{k+1}\| = |1-\|q_{k+1}\|| \leq \|q_{k+1}-y_{k+1}\| + \|y_{k+1}-z_{k+1}\| + 1-\|z_{k+1}\| < 3\gamma.
		$$
	Hence, combining the latter estimate with \eqref{eq:triangle} yields \eqref{eq:v-close}.
	
	By \eqref{eq:almost-daugavet}, we have \(\|x_k+\theta_{k+1}z_{k+1}\|>2-\gamma\). 
	Hence, by Hahn--Banach theorem, after multiplying by a suitable unimodular scalar, if necessary, there exists \(f\in S_{X^*}\) such that $f(x_k+\theta_{k+1}z_{k+1})=\|x_k+\theta_{k+1}z_{k+1}\|$.
	In particular,
		\begin{equation}\label{eq:norming}
			\operatorname{Re}f(x_k+\theta_{k+1}z_{k+1}) > 2-\gamma.
		\end{equation}
	Since \(\|x_k\|=1\), \(\|z_{k+1}\|\le1\), \(\|f\|=1\) and $|\theta_{k+1}|=1$, it follows from \eqref{eq:norming} that
		\begin{equation}\label{eq:real-estimates}
			\operatorname{Re}f(x_k)>1-\gamma \qquad\text{and}\qquad \operatorname{Re}f(\theta_{k+1}z_{k+1})>1-\gamma.
		\end{equation}
	Note that
		\begin{equation}\label{eq:omega-v-k+1}
			\operatorname{Re}f(\theta_{k+1}v_{k+1}) = \operatorname{Re}f(\theta_{k+1}z_{k+1}) + \operatorname{Re}f(\theta_{k+1}(v_{k+1}-z_{k+1})).
		\end{equation}
	Since \(\|f\|=1\) and \(|\theta_{k+1}|=1\), using \eqref{eq:v-close} we obtain
		\[
		\operatorname{Re}f(\theta_{k+1}(z_{k+1}-v_{k+1})) \leq \left|f(\theta_{k+1}(v_{k+1}-z_{k+1}))\right| \le \|v_{k+1}-z_{k+1}\| < 5\gamma.
		\]
	Combining the latter with \eqref{eq:real-estimates} and \eqref{eq:omega-v-k+1}, we deduce that
		\begin{equation}\label{eq:real-v}
			\operatorname{Re}f(\theta_{k+1}v_{k+1}) > 1-6\gamma.
		\end{equation}
	Then, by \eqref{eq:def-xk}, \eqref{eq:real-estimates}, and \eqref{eq:real-v},
		\begin{align*}
			\|S_{k+1}\| & = \|S_k+\theta_{k+1}v_{k+1}\| \\
			& \ge \operatorname{Re}f(S_k+\theta_{k+1}v_{k+1}) \\
			& = \|S_k\|\operatorname{Re}f(x_k) + \operatorname{Re}f(\theta_{k+1}v_{k+1}) \\
			& > (1-\gamma)\|S_k\|+1-6\gamma.
		\end{align*}
	Since \(\|S_k\|\le k\), it follows that $\|S_{k+1}\| > \|S_k\|+1-\gamma(k+6)$.
	Iterating the previous inequality from \(S_1\), we obtain
		\[
		\|S_m\|	> m-\gamma\sum_{j=1}^{m-1}(j+6).
		\]
	As
		\[
		\sum_{j=1}^{m-1}(j+6)\le m^2+6m,
		\]
	the choice of \(\gamma\) yields
		\[
		\|S_m\| > m-\gamma(m^2+6m) > (1-\varepsilon)m.
		\]
	Therefore,
		\[
		\left\|\sum_{j=1}^m\theta_jv_j\right\| > (1-\varepsilon)m.
		\]
	This completes the proof.
\end{proof}

\medskip
		
We will now show that this implies the existence of asymptotic \(\ell_1\)-structures.
		

\begin{proposition}\label{prop:super-adp-l1-blocks}
	Let $K\geq1$, and let $X$ be a Banach space with a $K$-unconditional basis.
	If \(X\) has the super ADP, then for every \(m\in\mathbb N\) and $0<\varepsilon<1/m$, there exist successive blocks $u_1,\dots,u_m\in S_X$ such that
		$$
		\left\|\sum_{j=1}^m u_j\right\| > (1-\varepsilon)m
		$$
	and, for every $a_1,\ldots,a_m\in \mathbb K$,
		$$
		\frac{1-m\varepsilon}{K} \sum_{j=1}^m |a_j| \leq \left\|\sum_{j=1}^m a_j u_j\right\| \leq \sum_{j=1}^m |a_j|.
		$$
	In particular, for every $\eta>0$, there exist successive normalized blocks which are $K(1+\eta)$-equivalent to the canonical basis of $\ell_1^m$.
\end{proposition}
	
\begin{proof}
	Fix $m\in\mathbb N$ and $0<\varepsilon<1/m$.
	
	By Lemma~\ref{lem:super-adp-linear-growth}, there exist successive blocks $v_1,\dots,v_m\in S_X$ and scalars $\theta_1,\dots,\theta_m\in\mathbb T$ such that
		\begin{equation}\label{eq:l1-growth}
			\left\|\sum_{j=1}^m\theta_jv_j\right\| > (1-\varepsilon)m.
		\end{equation}
	Define $u_j:=\theta_jv_j$, for every $j\in \{1,\ldots,m\}$.
	Then, \(u_1,\dots,u_m \in S_X\) are successive blocks and \eqref{eq:l1-growth} becomes
		\begin{equation}\label{eq:u-growth}
			\left\|\sum_{j=1}^m u_j\right\| > (1-\varepsilon)m.
		\end{equation}
	By Hahn--Banach theorem, there exists \(f\in S_{X^*}\) such that
		\[
		\operatorname{Re} f\left(\sum_{j=1}^m u_j\right) > (1-\varepsilon)m.
		\]
	Since \(\|u_j\|=1\) and \(\|f\|=1\), we have $\operatorname{Re}f(u_j) \le 1$, for every $j\in \{1,\ldots,m\}$.
	
	We claim that
		\begin{equation}\label{eq:functional-lower}
			\operatorname{Re}f(u_j)>1-m\varepsilon,
		\end{equation}
	for every $j\in \{1,\ldots,m\}$.
	Indeed, by means of contradiction, assume that there exists \(j_0\in\{1,\dots,m\}\) such that $\operatorname{Re}f(u_{j_0}) \le 1-m\varepsilon$.
	Then,
		\[
		\operatorname{Re}f\left( \sum_{j=1}^m u_j\right) = \sum_{j=1}^m\operatorname{Re}f(u_j) \le m-m\varepsilon = (1-\varepsilon)m,
		\]
	contradicting the choice of \(f\). 
	Hence, \eqref{eq:functional-lower} holds.
	
	Let now \(a_1,\dots,a_m\) be arbitrary scalars. 
	On the one hand, since $\|u_j\|=1$ for every $j\in \{1,\ldots,m\}$, we have
		$$
		\left\|\sum_{j=1}^m a_j u_j\right\| \leq \sum_{j=1}^m |a_j|.
		$$
	On the other hand, since the vectors $u_1,\ldots,u_m$ have disjoint supports, the vectors 
		$$
		\sum_{j=1}^m a_ju_j \qquad \text{and} \qquad \sum_{j=1}^m |a_j|u_j
		$$
	have coordinate coefficients with the same moduli. 
	Therefore, by the $K$-unconditionality of the basis, 
		$$
		\left\|\sum_{j=1}^m |a_j|u_j\right\| \leq K \left\|\sum_{j=1}^m a_j u_j\right\|.
		$$
	Thus, by \eqref{eq:functional-lower}, we have
		$$
		\left\|\sum_{j=1}^m a_j u_j\right\| \geq \frac{1}{K} \left\|\sum_{j=1}^m |a_j|u_j\right\| \geq \frac{1}{K} \operatorname{Re} f\left( \sum_{j=1}^m |a_j|u_j \right) = \frac{1}{K} \sum_{j=1}^m |a_j| \operatorname{Re} f(u_j) \geq \frac{1-m\varepsilon}{K} \sum_{j=1}^m |a_j|.
		$$

	Finally, for the additional part, let $\eta>0$ and fix 
		$$
		0<\varepsilon<\frac{\eta}{m(1+\eta)}<\frac{1}{m}.
		$$
	Then,
		$$
		1-m\varepsilon > 1-\frac{\eta}{1+\eta} = \frac{1}{1+\eta}.
		$$
	Applying the first part of the proposition with this $\varepsilon$, we obtain successive normalized blocks $u_1,\ldots,u_m$ which are $K(1+\eta)$-equivalent to the canonical basis of $\ell_1^m$.
\end{proof}


\section{Finite-dimensional obstructions}\label{sec:finite-obstructions}

We will develop the finite-dimensional obstruction which will contradict the local extremal behaviour required by the super ADP.

For every \(1\le p<\infty\) and $r>0$, define
	\[
	M_p(r) := \sup_{\theta\in\mathbb T} \left[\frac12\left(|r+\theta|-1\right)_+^p + \frac12\left(|r-\theta|-1\right)_+^p\right]^{1/p} \geq 0,
	\]
where \(t_+:=\max\{t,0\}\).


Before we continue, let us prove the following technical lemma regarding $M_p(r)$ and $M_1(r)$.

\begin{lemma}\label{lem:technical}
	For every $1\leq p<\infty$ and $r>0$, 
		$$
		M_1(r)\leq M_p(r) \leq 2^{1-1/p}M_1(r).
		$$
	Moreover,
		$$
		M_1(r) = \begin{cases}
			\dfrac{1}{2}[r+(r-2)_+], & \text{if } \mathbb K= \mathbb R,\\
			\max\left\{ \dfrac{r}{2},\sqrt{1+r^2}-1 \right\}, & \text{if } \mathbb K= \mathbb C.
		\end{cases}
		$$
	Furthermore, in both the real and complex cases, for every $0< r\leq 4/3$, we have 
		$$
		M_1(r)=\frac{r}{2},
		$$
	and therefore
		$$
		M_p(r) \leq 2^{1-1/p} \frac{r}{2}.
		$$
\end{lemma}

\begin{proof}
	We begin by proving the first part of the lemma.
	Fix $\theta \in \mathbb T$ and define $a_\theta := \left(|r+\theta|-1\right)_+$ and $b_\theta := \left(|r-\theta|-1\right)_+$.
	On the one hand, by Jensen's inequality
		$$
		\left(\frac{a_\theta+b_\theta}{2}\right)^p \leq \frac{a_\theta^p+b_\theta^p}{2}.
		$$
	Thus,
	$$
	\frac{a_\theta+b_\theta}{2} \leq \left(\frac{a_\theta^p+b_\theta^p}{2}\right)^{1/p} = 2^{-1/p} (a_\theta^p+b_\theta^p)^{1/p} \leq 2^{-1/p} (a_\theta+b_\theta) = 2^{1-1/p} \frac{a_\theta+b_\theta}{2}.
	$$
	Taking the supremum over $\theta\in\mathbb T$ yields the desired chain of inequalities.
	
	We next calculate $M_1(r)$. 
	Assume first that $\mathbb{K}=\mathbb{R}$. 
	Since $\mathbb{T}=\{-1,1\}$, we obtain, for either choice of $\theta$,
		$$
		\left\{\left(|r+\theta|-1\right)_+,\left(|r-\theta|-1\right)_+\right\} = \left\{ r,(r-2)_+\right\}.
		$$
	Therefore,
		$$
		M_1(r) = \frac12\left[r+(r-2)_+\right].
		$$
	
	Assume now that $\mathbb{K}=\mathbb{C}$. 
	Fix $\theta\in\mathbb{T}$ and set $s:=|r+\theta|$ and $t:=|r-\theta|$.
	Replacing $\theta$ by $-\theta$ if necessary, we may assume that $s\geq t$.
	
	Assume first that $t\leq1$. 
	Then, $(t-1)_+=0$ and, since $s\leq r+1$, we obtain
		$$
		\frac12\left[(s-1)_+ + (t-1)_+\right] = \frac12(s-1)_+ \leq \frac r2.
		$$
	
	Assume now that $t>1$, and so $s>1$.
	It is easy to show that $s^2+t^2 = 2(1+r^2)$.
	Hence, by the Cauchy--Schwarz inequality, 
		$$
		\frac12\left[(s-1)_+ + (t-1)_+\right] = \frac{s+t-2}{2} \leq \frac{\sqrt{2(s^2+t^2)}-2}{2} = \sqrt{1+r^2}-1
		$$
	Thus,
	$$
	M_1(r)
	\leq
	\max\left\{
	\frac r2,
	\sqrt{1+r^2}-1
	\right\}.
	$$
	
	For the reverse inequality, taking $\theta=i$ gives $|r+i|=|r-i|=\sqrt{1+r^2}$, and therefore $M_1(r)\geq\sqrt{1+r^2}-1$.
	If $0<r\leq2$, taking $\theta=1$ gives $\left(|r+1|-1\right)_+=r$ and $\left(|r-1|-1\right)_+=0$.
	Hence, $M_1(r)\geq r/2$.
	If $r>2$, then $\sqrt{1+r^2}-1>r/2$, so the value obtained with $\theta=i$ already equals the maximum of the two quantities. 
	Therefore, in the complex case, we have the desired equality.
	
	Finally, in the complex case, it is easy to prove that $\sqrt{1+r^2}-1\leq r/2$ if and only if $0<r\leq 4/3$.
	In the real case, we have $M_1(r)=r/2$ for every $0<r\leq2$. 
	Consequently, in both the real and complex cases, $M_1(r)=r/2$ whenever $0<r\leq 4/3$.
	Combining this with the first part of the proof, we conclude that
		$$
		M_p(r) \leq 2^{1-\frac1p}M_1(r)	= 2^{1-\frac1p}\frac r2 = 2^{-1/p}r
		$$
	for every $0<r\leq4/3$.
\end{proof}

\begin{remark}
	It is possible to obtain an explicit formula for $M_p(r)$ in both the real and complex cases, analogous to the formula for $M_1(r)$ in Lemma~\ref{lem:technical}. 
	However, the inequalities relating $M_1(r)$ and $M_p(r)$ show that, for the estimates needed below, it is enough to study $M_1(r)$.
\end{remark}

The following lemma provides a technical obstruction for a Banach space with $K$-unconditional basis to have the super ADP based on $M_p(r)$.
In order to state it, let us consider the following.
Let $1\leq p<\infty$, let $n\in\mathbb{N}$ be even, and let $u_1,\ldots,u_n$ be successive normalized block vectors. 
Choose $\varepsilon_1,\ldots,\varepsilon_n\in\{-1,1\}$ such that
	$$
	\sum_{j=1}^n \varepsilon_j=0,
	$$
and define
	$$
	x_0:=\frac{1}{n^{1/p}}\sum_{j=1}^n u_j \qquad \text{and} \qquad y_0:=\frac{1}{n^{1/p}}\sum_{j=1}^n \varepsilon_j u_j.
	$$
Set
	$$
	a:=\|x_0\|, \qquad b:=\|y_0\|, \qquad \text{and} \qquad r:=\frac{b}{a}.
	$$
	
\begin{lemma} \label{lem:local-obstruction-finite-p}
	With the above notation, let $K\geq 1$, and let $X$ be a Banach space with a $K$-unconditional basis. Assume that, for some $C\geq 1$, the successive normalized block vectors $u_1,\ldots,u_n$ satisfy
		\begin{equation}\label{eq:equivalent}
			\left\|\sum_{j=1}^n \lambda_j u_j\right\| \leq C \left(\sum_{j=1}^n |\lambda_j|^p\right)^{1/p}
		\end{equation}
	for every choice of scalars $\lambda_1,\ldots,\lambda_n$.
	If
		\begin{equation}\label{eq:obstruction-condition}
			K+\frac{C}{b}M_p(r)<2,
		\end{equation}
	then \(X\) does not have the super ADP.
	In particular, if $0<r\leq 4/3$ and $K+\frac{C}{2^{1/p}a}<2$, then $X$ does not have the super ADP.
\end{lemma}

\begin{proof}
	Since the vectors $u_1,\ldots,u_n$ have pairwise disjoint supports and are nonzero, we have $x_0\neq 0$ and $y_0\neq 0$.
	Thus, $a>0$ and $b>0$. 
	Define
		$$
		x:=\frac{x_0}{a}\in S_X \qquad \text{and} \qquad \widetilde{y}_0:=\frac{y_0}{b}\in S_X.
		$$
	Also, let
		$$
		A:=\bigcup_{j=1}^n\operatorname{supp}(u_j).
		$$
	Since the vectors $u_1,\ldots,u_n$ are block vectors, the set $A$ is finite.
	
	By \eqref{eq:obstruction-condition}, we may choose $\delta>0$ sufficiently small so that
		$$
		K(1+\delta)+\frac{C}{b}M_p(r)+\delta<2.
		$$
	Consider the set
		$$
		W:= \left\{ y\in B_X \colon \left\|P_Ay-\widetilde{y}_0\right\|<\delta \right\}.
		$$
	Since $P_A$ has finite rank, it is weak-to-norm continuous. 
	Hence, $W$ is a relatively weakly open subset of $B_X$. 
	Moreover, $P_A\widetilde{y}_0=\widetilde{y}_0$, and therefore $\widetilde{y}_0\in W\cap S_X$.
	In particular, $W\cap S_X\neq\varnothing$.
	
	Fix $y\in W$, and define $h:=P_Ay-\widetilde{y}_0$ and $w:=(I-P_A)y$.
	Then, $\|h\|<\delta$, and $y=\widetilde{y}_0+h+w$.
	Moreover, $\widetilde{y}_0+w=y-h$, and hence
		\begin{equation}\label{eq:y_tilde_w}
			\|\widetilde{y}_0+w\| \leq \|y\|+\|h\| < 1+\delta.
		\end{equation}
	
	Fix now $\theta\in\mathbb{T}$. Since $r=b/a$, we have
		$$
		x+\theta\widetilde{y}_0 = \frac{1}{an^{1/p}}\sum_{j=1}^n u_j + \frac{\theta}{bn^{1/p}}\sum_{j=1}^n\varepsilon_ju_j = \frac{1}{bn^{1/p}} \sum_{j=1}^n \left(r+\theta\varepsilon_j\right)u_j.
		$$
	For every $j\in\{1,\ldots,n\}$, set $\alpha_j:=r+\theta\varepsilon_j$.
	We decompose $\frac{\alpha_j}{bn^{1/p}} = \beta_j+c_j$ as follows.
	 If $\alpha_j\neq0$, define
		$$
		\beta_j := \frac{\alpha_j}{|\alpha_j|} \frac{\min\{|\alpha_j|,1\}}{bn^{1/p}} \qquad \text{and} \qquad c_j := \frac{\alpha_j}{bn^{1/p}}-\beta_j.
		$$
	If $\alpha_j=0$, put $\beta_j=c_j=0$.
	Then,
		\begin{equation}\label{eq:beta_j-c_j}
			|\beta_j| \leq \frac{1}{bn^{1/p}} \qquad \text{and} \qquad |c_j| = \frac{\bigl(|r+\theta\varepsilon_j|-1\bigr)_+}{bn^{1/p}}.
		\end{equation}
	Consequently,
		$$
		x+\theta\widetilde{y}_0 = \sum_{j=1}^n\beta_ju_j + \sum_{j=1}^n c_ju_j,
		$$
	and therefore
		\begin{equation}\label{eq:estimate}
			x+\theta(\widetilde{y}_0+w) = \left(\sum_{j=1}^n \beta_ju_j+\theta w\right) + \sum_{j=1}^n c_ju_j.
		\end{equation}
	
	We first estimate the first term of \eqref{eq:estimate}. Recall that
		$$
		\widetilde{y}_0+w =\frac{1}{bn^{1/p}} \sum_{j=1}^n\varepsilon_ju_j+w.
		$$
	On each support $\operatorname{supp}(u_j)$, by \eqref{eq:beta_j-c_j} the coordinate coefficients of $\beta_ju_j$ are dominated in modulus by the corresponding coordinate coefficients of $\frac{\varepsilon_j}{bn^{1/p}}u_j$.
	Outside $A$, the vectors $\theta w$ and $w$ have coordinate coefficients with the same moduli. Hence, by the $K$-unconditionality of the basis,
		$$
		\left\| \sum_{j=1}^n\beta_ju_j+\theta w \right\| \leq K \left\| \widetilde{y}_0+w \right\|.
		$$
	Indeed, if $w$ is not finitely supported, this inequality follows by applying the definition of $K$-unconditionality to the corresponding finite partial sums and then passing to the norm limit. 
	Thus, by \eqref{eq:y_tilde_w},
		\begin{equation}\label{eq:first_estimate}
			\left\|	\sum_{j=1}^n\beta_ju_j+\theta w \right\| < K(1+\delta).
		\end{equation}
	
	We next estimate the second term of \eqref{eq:estimate}. 
	By \eqref{eq:equivalent} and \eqref{eq:beta_j-c_j}, we have
		$$
		\left\| \sum_{j=1}^n c_ju_j \right\| \leq C\left( \sum_{j=1}^n|c_j|^p \right)^{1/p} = \frac{C}{bn^{1/p}} \left[\sum_{j=1}^n \left(|r+\theta\varepsilon_j|-1\right)_+^p\right]^{1/p}.
		$$
	By the choice of the $\varepsilon_j$'s, exactly $n/2$ of the signs $\varepsilon_j$ are equal to $1$, and exactly $n/2$ are equal to $-1$. 
	Therefore, by the definition of $M_p(r)$, it follows that 
		\begin{equation}\label{eq:second_estimate}
			\left\| \sum_{j=1}^n c_ju_j \right\| \leq \frac{C}{b} \left[ \frac{1}{2}\left(|r+\theta|-1\right)_+^p + \frac{1}{2}\left(|r-\theta|-1\right)_+^p \right]^{1/p} \leq \frac{C}{b} M_p(r).
		\end{equation}
	
	Combining \eqref{eq:first_estimate} and \eqref{eq:second_estimate}, we obtain
		\begin{equation}\label{eq:combined_estimates}
			\left\| x+\theta(\widetilde{y}_0+w) \right\| < K(1+\delta) + \frac{C}{b}M_p(r).
		\end{equation}
	Finally, since $y=\widetilde{y}_0+h+w$, we have $x+\theta y = x+\theta(\widetilde{y}_0+w)+\theta h$.
	Therefore, by \eqref{eq:combined_estimates} and the fact that $\|h\|<\delta$,
		\begin{equation}\label{eq:final_estimate}
			\|x+\theta y\| \leq \left\|x+\theta(\widetilde{y}_0+w)\right\| + \|h\| < K(1+\delta) + \frac{C}{b}M_p(r) + \delta <2.
		\end{equation}
	Set
		$$
		\rho := K(1+\delta) + \frac{C}{b}M_p(r) + \delta.
		$$
	Then $\rho<2$, and \eqref{eq:final_estimate} shows that $\|x+\theta y\|<\rho$ for every $y\in W$ and $\theta\in\mathbb{T}$. Hence,
		$$
		\sup_{y\in W} \max_{\theta\in\mathbb{T}} \|x+\theta y\| \leq \rho < 2.
		$$
	Since $x\in S_X$ and $W$ is a relatively weakly open subset of $B_X$ satisfying $W\cap S_X\neq\varnothing$, this contradicts the super ADP.
	
	For the additional part, by Lemma~\ref{lem:technical}, $M_p(r)\leq 2^{-1/p}r$.
	Since $r=b/a$, it follows that
		$$
		\frac{C}{b}M_p(r) \leq \frac{C}{b}2^{-1/p}r = \frac{C}{b}2^{-1/p}\frac{b}{a} = \frac{C}{2^{1/p}a}.
		$$
	Therefore,
		$$
		K+\frac{C}{b}M_p(r) \leq K+\frac{C}{2^{1/p}a} < 2.
		$$
	The first part of the lemma now implies that $X$ does not have the super ADP.
\end{proof}

We are now ready to prove Theorem~\ref{thm:intro-main}.

\medskip

%

\textit{Proof of Theorem~\ref{thm:intro-main}.}
	Assume, by means of contradiction, that $X$ has a $K$-unconditional basis and the super ADP for some $1\leq K<3/2$.
	Fix an even $n\in\mathbb{N}$. 
	Since $K<3/2$, we may choose
		$$
		0<\varepsilon< \min\left\{ \frac{1}{n},\ \frac{1}{4},\ \frac{3-2K}{2(2-K)}\right\}.
		$$
	By Proposition~\ref{prop:super-adp-l1-blocks}, there exist successive normalized block vectors $u_1,\ldots,u_n\in S_X$ such that
		$$
		\left\|\sum_{j=1}^n u_j\right\| > (1-\varepsilon)n
		$$
	and
			\begin{equation}\label{eq:p-1_C-1}
			\left\|\sum_{j=1}^n\lambda_j u_j\right\| \leq \sum_{j=1}^n|\lambda_j|
		\end{equation}
	for every choice of scalars $\lambda_1,\ldots,\lambda_n$.
	
	Choose $\varepsilon_1,\ldots,\varepsilon_n\in\{-1,1\}$ such that
		$$
		\sum_{j=1}^n\varepsilon_j=0,
		$$
	and define, with $p=1$,
		$$
		x_0:=\frac{1}{n}\sum_{j=1}^n u_j \qquad \text{and} \qquad y_0:=\frac{1}{n}\sum_{j=1}^n\varepsilon_j u_j.
		$$
	Set
		$$
		a:=\|x_0\|, \qquad b:=\|y_0\|, \qquad \text{and} \qquad r:=\frac{b}{a}.
		$$
	The choice of the vectors $u_1,\ldots,u_n$ gives
		$$
		a = \frac{1}{n} \left\|\sum_{j=1}^n u_j\right\| > 1-\varepsilon.
		$$
	Moreover, since the vectors $u_j$ are normalized,
		$$
		b = \left\|\frac{1}{n} \sum_{j=1}^n\varepsilon_j u_j\right\| \leq \frac{1}{n} \sum_{j=1}^n\|u_j\| = 1.
		$$
	Consequently,
		$$
		0 < r = \frac{b}{a} < \frac{1}{1-\varepsilon} < \frac{4}{3},
		$$
	where the last inequality follows from $\varepsilon<1/4$.
	
	The estimate \eqref{eq:p-1_C-1} shows that the hypothesis of Lemma~\ref{lem:local-obstruction-finite-p} is satisfied with $p=1$ and $C=1$.
	Furthermore,
		$$
		K+\frac{1}{2a} < K+\frac{1}{2(1-\varepsilon)}.
		$$
	Since
		$$
		\varepsilon < \frac{3-2K}{2(2-K)} = 1-\frac{1}{2(2-K)},
		$$
	it is easy to show that
		$$
		K+\frac{1}{2a} < 2.
		$$
	The additional part of Lemma~\ref{lem:local-obstruction-finite-p} now implies that $X$ does not have the super ADP, contradicting our initial assumption.
\qed

\medskip

This settles the $1$-unconditional part of \cite[Question~6.4]{Langemets2025} and, more generally, its $K$-unconditional case for $1\leq K<3/2$.



\section{Countable weak $\pi$-bases}\label{sec:pibase}

We now turn to the weak topological structure of bounded convex subsets of Banach spaces with bases.
We begin by providing a positive answer to \cite[Question~5.1]{LooPerreau2026}.

\begin{theorem}\label{thm:pi-base-ball}
	Let $X$ be a Banach space with a $1$-unconditional Schauder basis. Then, $(B_X,w)$ has a countable $\pi$-base.
\end{theorem}

To prove Theorem~\ref{thm:pi-base-ball}, it is enough to show that every nonempty relatively weakly open subset of $B_X$ contains a nonempty open subset for the topology of coordinatewise convergence. 
For completeness, we recall the definition of this topology below.

Let $X$ be a Banach space with a Schauder basis $(e_j)_{j=1}^{\infty}$, and let $(e_j^*)_{j=1}^{\infty}$ be the associated biorthogonal coordinate functionals. 
Consider the map $\Phi:B_X\longrightarrow \mathbb K^{\mathbb N}$ defined by 
	$$
	\Phi(x):=(e_n^*(x))_{n=1}^{\infty}.
	$$
We denote by $\tau_c$ the topology induced on $B_X$ by the product topology of $\mathbb K^{\mathbb N}$ through $\Phi$. 
A convenient neighborhood basis for $\tau_c$ is given by the sets
	$$
	U(x,N,r) := \left\{ z\in B_X \colon \|P_N(z-x)\|<r\right\},
	$$
where $x\in B_X$, $N\in\mathbb N$ and $r>0$.
Since $P_N$ has finite rank, it is weak-to-norm continuous. 
Hence, every set $U(x,N,r)$ is relatively weakly open in $B_X$, and therefore $\tau_c\subset w|_{B_X}$.
Moreover, since $\mathbb K^{\mathbb N}$ is second countable in its product topology, the topology $\tau_c$ induced on $B_X$ through $\Phi$ is also second countable.
Consequently, if every nonempty relatively weakly open subset of $B_X$ contains a nonempty $\tau_c$-open subset, then any countable base for $\tau_c$ is a countable $\pi$-base for $(B_X,w)$.

Before proving Theorem~\ref{thm:pi-base-ball}, we establish a stabilization lemma that contains the main part of the argument. 
Its proof combines a gliding-hump construction with coordinatewise unimodular multipliers.

\begin{lemma}\label{lem:stabilization}
	Let $X$ be a Banach space with a $1$-unconditional basis $(e_j)_{j=1}^{\infty}$, let $x\in B_X\cap c_{00}$, let $O$ be a $\tau_c$-neighborhood of $x$, and let $g_1,\ldots,g_m\in B_{X^*}$.
	Fix $n\in\{1,\ldots,m\}$ and let $a,b>0$. 
	Then, there exist $y\in B_X\cap c_{00}$ and a $\tau_c$-neighborhood $V$ of $y$ such that $V\subset O$, $|g_i(y-x)|<a$ for every $i\in\{1,\ldots,m\}$, and $\operatorname{diam}g_n(V)<b$, where $\operatorname{diam}g_n(V) := \sup_{u,v\in V} |g_n(u)-g_n(v)|$.
\end{lemma}

\begin{proof}
	Set $f:=g_n$ and define
		$$
		L := \inf \left\{ \sup_{z\in U}\operatorname{Re}f(z) \colon U\text{ is a }\tau_c\text{-neighborhood of }x \right\}.
		$$
	Since $x$ belongs to every $\tau_c$-neighborhood of itself, we have $L\geq \operatorname{Re}f(x)$.
	
	Choose 
		$$
		0<\alpha<\frac{b}{8}.
		$$
	By the definition of $L$, there exists a $\tau_c$-neighborhood $U'$ of $x$ such that $\sup_{z\in U'}\operatorname{Re}f(z)<L+\alpha$.
	Since $O\cap U'$ is a $\tau_c$-neighborhood of $x$, we may choose $N_0\in\mathbb N$ and $r>0$ such that $\operatorname{supp}(x)\subset\{1,\ldots,N_0\}$ and 
		$$
		U_0 := U(x,N_0,r) \subset O\cap U'.
		$$
	Therefore, 
		\begin{equation}\label{eq:sup-alpha}
			\sup_{z\in U_0}\operatorname{Re}f(z)<L+\alpha.
		\end{equation}
	
	Choose
		$$
		0 < \gamma < \min \left\{\frac{a}{2},\ r\right\}.
		$$
	Since $\tau_c$ is second countable, we may fix a decreasing countable local base $(W_k)_{k=1}^{\infty}$ at $x$ for $\tau_c$ such that $W_k\subset U_0$ for every $k\in\mathbb N$. 
	As $(W_k)_{k=1}^{\infty}$ is a local base at $x$, the definition of $L$ yields
		$$
		\inf_{k\in\mathbb N}\sup_{z\in W_k}\operatorname{Re}f(z)=L.
		$$
	For every $k\in\mathbb N$, choose $\widehat z_k\in W_k$ such that $\operatorname{Re}f(\widehat z_k)>L-\frac{1}{2k}$.
	Since $P_N \widehat z_k\rightarrow \widehat z_k$ in norm as $N\to \infty$, we may choose $M_k\in\mathbb N$ sufficiently large so that $P_{M_k} \widehat z_k\in W_k$ and $\left|f(P_{M_k} \widehat z_k- \widehat z_k)\right| < \frac{1}{2k}$.
	Notice that $P_{M_k}\widehat z_k\in B_X$, since the canonical projections are contractive. 
	Define $z_k:=P_{M_k}\widehat z_k$.
	Then, $z_k\in W_k\cap c_{00}$ and $\operatorname{Re}f(z_k)>L-\frac{1}{k}$.
	Since $z_k\in W_k$ and $(W_k)_{k=1}^{\infty}$ is a decreasing local base at $x$, we have $z_k\to x$ in the topology $\tau_c$, that is, coordinatewise.
	Moreover, the sequence $\left(g_1(z_k),\ldots,g_m(z_k)\right)_{k=1}^{\infty}$ is bounded in the finite-dimensional space $\mathbb K^m$. 
	Passing to a subsequence, if necessary, we may assume that $\left(g_1(z_k),\ldots,g_m(z_k)\right) \to (q_1,\ldots,q_m)$ for some $(q_1,\ldots,q_m)\in\mathbb K^m$.
	
	Fix a sufficiently far term of this subsequence, which we will call $z_1$, such that 
		\begin{equation}\label{eq:P_N_0-z_1}
			\|P_{N_0}(z_1-x)\|<\gamma,
		\end{equation}
		\begin{equation}\label{eq:q_i-z_1}
			\max_{1\leq i\leq m} |g_i(z_1)-q_i| < \gamma,
		\end{equation}
	and 
		\begin{equation}\label{eq:Re_f_z-1_alpha}
			\operatorname{Re}f(z_1)>L-\alpha
		\end{equation}
	Now fix $N_1\geq N_0$ such that $\operatorname{supp}(z_1)\subset\{1,\ldots,N_1\}$.
	Also, choose a later term of the subsequence, which we denote by $z_2$, such that
		\begin{equation}\label{eq:P_N_1-z_2}
			\|P_{N_1}(z_2-x)\|<\gamma
		\end{equation}
		\begin{equation}\label{eq:q_i-z_2}
			\max_{1\leq i\leq m} |g_i(z_2)-q_i| <\gamma
		\end{equation}
	and 
		\begin{equation}\label{eq:Re_f_z-2_alpha}
			\operatorname{Re}f(z_2)>L-\alpha
		\end{equation}
	And fix $N_2>N_1$ such that $\operatorname{supp}(z_2)\subset \{1,\ldots,N_2\}$.
	Define 
		$$
		R_{N_1}:=2P_{N_1}-I.
		$$
	The operator $R_{N_1}$ fixes the first $N_1$ coordinates and changes the signs of all the remaining coordinates. 
	Hence, $R_{N_1}$ is an isometry in both the real and complex cases.
	
	Set
		\begin{equation}\label{eq:y_z-1_z-2}
			y := \frac{1}{2}z_1 + \frac{1}{2}R_{N_1}z_2.
		\end{equation}
	Since $z_1,R_{N_1}z_2\in B_X$, we have $y\in B_X$.
	Moreover, $y\in c_{00}$.
	
	For every $i\in\{1,\ldots,m\}$, we have
		$$
		g_i(y-x) = \frac{1}{2} g_i(z_1) + \frac{1}{2}g_i(R_{N_1}z_2) - g_i(x) = \frac{1}{2}g_i(z_1-z_2) + g_i(P_{N_1}z_2-x).
		$$
	By \eqref{eq:q_i-z_1} and \eqref{eq:q_i-z_2},
		$$
		|g_i(z_1-z_2)|<2\gamma.
		$$
	Since $x=P_{N_1}x$, it follows from \eqref{eq:P_N_1-z_2} and $\|g_i\|\leq1$ that
		$$
		|g_i(P_{N_1}z_2-x)| = |g_i(P_{N_1}(z_2-x))| \leq \|P_{N_1}(z_2-x)\| < \gamma.
		$$
	Thus, for every $i\in\{1,\ldots,m\}$,
		$$
		|g_i(y-x)|<2\gamma<a.
		$$
	
	Since $N_0\leq N_1$, we have $P_{N_0}R_{N_1}=P_{N_0}$.
	Consequently,
		$$
		P_{N_0}(y-x) = \frac{1}{2}P_{N_0}(z_1-x) + \frac{1}{2}P_{N_0}(z_2-x).
		$$
	By \eqref{eq:P_N_0-z_1}, \eqref{eq:P_N_1-z_2}, and the contractivity of $P_{N_0}$, we obtain
		$$
		\|P_{N_0}(y-x)\|<\gamma<r.
		$$
	Hence, $y\in U_0\subset O$.
	
	Choose
		$$
		0 < \delta < \min \left\{r-\gamma,\ \frac{b}{8}\right\}.
		$$
	Denote 
		$$
		V:=U(y,N_2,\delta).
		$$
	Then, $V$ is a $\tau_c$-neighborhood of $y$.
	If $w\in V$, then
		$$
		\|P_{N_0}(w-x)\| \leq \|P_{N_0}(w-y)\| + \|P_{N_0}(y-x)\| \leq \|P_{N_2}(w-y)\| + \|P_{N_0}(y-x)\| < \delta+\gamma < r.
		$$
	Thus, 
		$$
		V\subset U_0\subset O.
		$$
	
	It remains to prove that $\operatorname{diam}f(V)<b$.
	Fix $w\in V$ and define 
		$$
		w_3:=(I-P_{N_2})w.
		$$
	Choose $\sigma\in\mathbb T$ such that 
		\begin{equation}\label{eq:w_3}
			\operatorname{Re}\left(\sigma f(w_3)\right) = |f(w_3)|.
		\end{equation}
	Define
		$$
		\widetilde w := P_{N_1}w - (P_{N_2}-P_{N_1})w + \sigma w_3.
		$$
	The vector $\widetilde w$ is obtained from $w$ by a coordinatewise unimodular change: the coordinates up to $N_1$ are multiplied by $1$, those from $N_1+1$ to $N_2$ by $-1$, and those after $N_2$ by $\sigma$. Hence, $\|\widetilde w\|=\|w\|\leq1$.
	Thus, $\widetilde w\in B_X$.
	Moreover, $\widetilde w$ and $w$ have the same first $N_0$ coordinates. 
	Since $w\in V\subset U_0$, we obtain $\widetilde w\in U_0$.
	Therefore, by \eqref{eq:sup-alpha}, 
		\begin{equation}\label{eq:tilde-w_L+alpha}
			\operatorname{Re}f(\widetilde w)<L+\alpha.
		\end{equation}
	Define
		$$
		\widetilde y := P_{N_1}y - (P_{N_2}-P_{N_1})y.
		$$
	Since $z_1$ is supported in the first $N_1$ coordinates, $z_2$ is supported in the first $N_2$ coordinates, and $y$ is given by \eqref{eq:y_z-1_z-2}, we have
		$$
		\widetilde y = \frac{1}{2}z_1 + \frac{1}{2}z_2.
		$$
	Consequently, by \eqref{eq:Re_f_z-1_alpha}, \eqref{eq:Re_f_z-2_alpha}, and the previous equality
		\begin{equation}\label{eq:tilde-y_L-alpha}
			\operatorname{Re}f(\widetilde y)>L-\alpha.
		\end{equation}
	The coordinatewise sign change which fixes the first $N_1$ coordinates and changes the signs of the coordinates from $N_1+1$ to $N_2$ is an isometry. Hence,
		$$
		\operatorname{Re}f(\widetilde y) - \operatorname{Re}f \left(P_{N_1}w-(P_{N_2}-P_{N_1})w\right) \leq \left\|P_{N_1}w-(P_{N_2}-P_{N_1})w-\widetilde y\right\| = \|P_{N_2}(w-y)\| < \delta.
		$$
	Using \eqref{eq:w_3}, the previous chain of inequalities and \eqref{eq:tilde-y_L-alpha}, we obtain
		$$
		\operatorname{Re}f(\widetilde w) = \operatorname{Re}f \left(P_{N_1}w-(P_{N_2}-P_{N_1})w\right)+|f(w_3)| > \operatorname{Re}f(\widetilde y)-\delta+|f(w_3)| > L-\alpha-\delta+|f(w_3)|.
		$$
	Combining \eqref{eq:tilde-w_L+alpha} and the previous chain of inequalities, we obtain 
		\begin{equation}\label{eq:w_3-alpha-delta}
			|f(w_3)|<2\alpha+\delta
		\end{equation}
	Since $y$ is supported in the first $N_2$ coordinates, $w-y=P_{N_2}(w-y)+w_3$.
	Therefore, by \eqref{eq:w_3-alpha-delta} and the fact that $w\in V$, we have
		$$
		|f(w-y)| \leq |f(P_{N_2}(w-y))|+|f(w_3)| \leq \|P_{N_2}(w-y)\|+|f(w_3)| < \delta+2\alpha+\delta = 2\alpha+2\delta.
		$$
	Since $w\in V$ is arbitrary, for every $u,v\in V$ we have
		$$
		|f(u)-f(v)| \leq |f(u)-f(y)|+|f(v)-f(y)| < 4\alpha+4\delta.
		$$
	By the choice of $\alpha$ and $\delta$, we have $4\alpha+4\delta<b$, finishing the proof.
\end{proof}

Now we can prove Theorem~\ref{thm:pi-base-ball}.

\medskip

\textit{Proof of Theorem~\ref{thm:pi-base-ball}.}
	Let $W$ be a nonempty relatively weakly open subset of $B_X$, and choose $x_0\in W$.
	There exist nonzero functionals $f_1,\ldots,f_m\in X^*$ and $\varepsilon_1,\ldots,\varepsilon_m>0$ such that
		$$
		W_0 := \left\{z\in B_X \colon |f_i(z-x_0)|<\varepsilon_i\text{ for every }i\in\{1,\ldots,m\}\right\} \subset W.
		$$
	For every $i\in\{1,\ldots,m\}$, define
		$$
		g_i:=\frac{f_i}{\|f_i\|} \in S_{X^*} \qquad \text{and} \qquad \eta_i:=\frac{\varepsilon_i}{\|f_i\|}>0.
		$$
	Then,
		$$
		W_0 = \left\{ z\in B_X \colon |g_i(z-x_0)|<\eta_i \text{ for every }i\in\{1,\ldots,m\}\right\}.
		$$
	
	Since $P_Nx_0\to x_0$ in norm, there exists $N\in\mathbb N$ such that
		$$
		x^{(0)}:=P_Nx_0\in W_0.
		$$
	Notice that $x^{(0)}\in B_X\cap c_{00}$.
	Set
		$$
		\rho := \min_{1\leq i\leq m} \left\{\eta_i-\left|g_i\left(x^{(0)}-x_0\right)\right|\right\}.
		$$
	Since $x^{(0)}\in W_0$, we have $\rho>0$.
	Choose
		$$
		0<\theta<\frac{\rho}{m+1}.
		$$
	
	We now apply Lemma~\ref{lem:stabilization} inductively. 
	Set $O^{(0)}:=B_X$.
	Assume that, for some $n\in\{1,\ldots,m\}$, we have already constructed a finitely supported vector $x^{(n-1)}\in B_X$ and a $\tau_c$-neighborhood $O^{(n-1)}$ of $x^{(n-1)}$.
	
	Apply Lemma~\ref{lem:stabilization} with $x=x^{(n-1)}$, $O=O^{(n-1)}$, the family $g_1,\ldots,g_m$, distinguished functional $g_n$, and $a=b=\theta$.
	We obtain a vector $x^{(n)}\in B_X\cap c_{00}$ and a $\tau_c$-neighborhood $O^{(n)}$ of $x^{(n)}$ such that 
		\begin{equation}\label{eq:x_n-O_n}
			x^{(n)}\in O^{(n)}\subset O^{(n-1)},
		\end{equation}
		\begin{equation}\label{eq_x_n-theta}
			\left|g_i\left(x^{(n)}-x^{(n-1)}\right)\right|<\theta
		\end{equation}
	for every $i\in\{1,\ldots,m\}$, and
		\begin{equation}\label{eq:diam_O-j}
			\operatorname{diam}g_n\left(O^{(n)}\right)<\theta.
		\end{equation}
	
	After completing the induction, define $V:=O^{(m)}$.
	Then, $V$ is a nonempty $\tau_c$-open subset of $B_X$.
	We claim that $V\subset W_0$.
	Fix $i\in\{1,\ldots,m\}$ and $z\in V$. 
	Since the neighborhoods are nested, we have $z\in O^{(m)}\subset O^{(i)}$.
	Moreover, by \eqref{eq:x_n-O_n}, $x^{(m)}\in O^{(m)}\subset O^{(i)}$.
	Therefore, by \eqref{eq:diam_O-j} applied at the $i$-th stage,
		$$
		\left|g_i\left(z-x^{(m)}\right)\right|<\theta.
		$$
	On the other hand, by \eqref{eq_x_n-theta},
		$$
		\left|g_i\left(x^{(m)}-x^{(0)}\right)\right| = \left|g_i \left(\sum_{j=1}^m \left(x^{(j)}-x^{(j-1)}\right)\right)\right| \leq \sum_{j=1}^m \left|g_i\left(x^{(j)}-x^{(j-1)}\right)\right| < m\theta.
		$$
	Thus,
		\begin{align*}
			|g_i(z-x_0)| & \leq \left|g_i\left(z-x^{(m)}\right)\right| + \left|g_i\left(x^{(m)}-x^{(0)}\right)\right| + \left|g_i\left(x^{(0)}-x_0\right)\right| \\
			& < \theta+m\theta+\eta_i-\rho \\
			&< \eta_i,
		\end{align*}
	where the last inequality follows from the choice of $\theta$. 
	Since this holds for every $i\in\{1,\ldots,m\}$, we conclude that $z\in W_0$.
	As $z\in V$ is arbitrary, $V\subset W_0\subset W$.
	Thus, every nonempty relatively weakly open subset of $B_X$ contains a nonempty $\tau_c$-open subset.
	
	Finally, let $\mathcal B=\{U_n:n\in\mathbb N\}$ be a countable base for $\tau_c$. 
	Since $\tau_c\subset w|_{B_X}$, every $U_n$ is relatively weakly open in $B_X$. 
	Let $W$ be a nonempty relatively weakly open subset of $B_X$. 
	By the above, there exists a nonempty $\tau_c$-open set $V$ such that $V\subset W$.
	Since $\mathcal B$ is a base for $\tau_c$, there exists $n\in\mathbb N$ such that $\varnothing\neq U_n\subset V\subset W$.
	This completes the proof.
\qed

\medskip

The preceding argument is specific to the unit ball and does not automatically extend to arbitrary bounded convex subsets of $X$. 
The crucial point is that $B_X$ is invariant under every coordinatewise multiplication by scalars of modulus one, an invariance that is essential in the stabilization procedure. 
A general bounded convex subset need not possess this property, and in fact the existence of a countable weak $\pi$-base is not inherited by bounded closed convex subsets. 
Indeed, in the binary tree space $X_T$ (see Section~\ref{sec:non-scd}), the unit ball $B_{X_T}$ has a countable weak $\pi$-base \cite[Theorem~1.6]{LooPerreau2026}, whereas its positive part $B_{X_T}^{+}$ is not SCD \cite[Theorem~3.6]{LooPerreau2026} and therefore cannot have a countable weak $\pi$-base \cite[Proposition~2.21]{AvilesKadetsMartinMeriShepelska2010}. 
Thus, additional hypotheses are required in order to obtain such a conclusion for every bounded convex subset. 
We show below that this is the case when the basis is shrinking or boundedly complete.


The proof combines two essentially different ideas.
For shrinking bases, the conclusion follows from the absence of isomorphic copies of $\ell_1$ together with the characterization of weak $\pi$-bases obtained in \cite{AvilesKadetsMartinMeriShepelska2010}.
For boundedly complete bases, the argument relies instead on duality, the Radon--Nikodým property and the convex point of continuity property.
These two mechanisms will later be abstracted into independent permanence results.

\begin{theorem}\label{thm:pi-base}
	Let $X$ be a Banach space with a Schauder basis.
	If the basis is shrinking or boundedly complete, then every bounded convex subset of $X$ has a countable $\pi$-base for its relative weak topology.
	In particular, $(B_X,w)$ has a countable $\pi$-base.
\end{theorem}

\begin{proof}
	Let $(e_j)_{j=1}^{\infty}$ be the Schauder basis of $X$, and let
	$(e_j^*)_{j=1}^{\infty}$ be the associated biorthogonal functionals.
	We distinguish the two cases.
	
	Assume first that $(e_j)_{j=1}^{\infty}$ is shrinking.
	Then $(e_j^*)_{j=1}^{\infty}$ is a Schauder basis of $X^*$, and therefore
	$X^*$ is separable.
	Consequently, $X$ contains no isomorphic copy of $\ell_1$.
	
	Let $A\subset X$ be bounded and convex.
	Since $X$ has a Schauder basis, it is separable, and hence $A$ is separable.
	Moreover, $A$ contains no sequence equivalent to the canonical basis of $\ell_1$, since such a sequence would generate an isomorphic copy of $\ell_1$ inside $X$.
	Therefore, by \cite[Theorem~2.22]{AvilesKadetsMartinMeriShepelska2010},
	$(A,w)$ has a countable $\pi$-base.
	
	Assume now that $(e_j)_{j=1}^{\infty}$ is boundedly complete.
	Set $Y := \overline{\operatorname{span}} \{e_n^*:n\in\mathbb N\} \subset X^*$.
	By \cite[Theorem~3.2.15]{AlbiacKalton},
	the sequence $(e_j^*)_{j=1}^{\infty}$ is a shrinking basis of $Y$, and the canonical embedding of $X$ into $Y^*$ is an isomorphism.
	Hence, $X$ is isomorphic to a dual Banach space.
	Since $X$ has a Schauder basis, it is separable.
	Therefore, $Y^*$ is separable.
	By \cite[Proposition~5.5.6]{AlbiacKalton},
	$Y^*$ has the Radon--Nikodým property.
	Since the Radon--Nikodým property is preserved under isomorphisms,
	$X$ also has the Radon--Nikodým property.
	
	Let again $A\subset X$ be bounded and convex.
	Its norm closure $\overline A:=\overline A^{\|\cdot\|}$ is separable, closed, bounded and convex.
	Since $X$ has the Radon--Nikodým property,
	$\overline A$ has the convex point of continuity property.
	Hence, by
	\cite[Corollary~6.4]{AvilesKadetsMartinMeriShepelska2010},
	$(\overline A,w)$ has a countable $\pi$-base,
	say $(U_n)_{n=1}^{\infty}$.
	Since $A$ is convex,
	its norm closure coincides with its weak closure.
	Thus, $A$ is weakly dense in $\overline A$, and every set $V_n:=U_n\cap A$ is nonempty and relatively weakly open in $A$.
	
	We claim that $(V_n)_{n=1}^{\infty}$ is a $\pi$-base for $(A,w)$. Let $W\subset A$ be a nonempty relatively weakly open subset.
	Then, there exists a weakly open subset $G\subset X$ such that $W=G\cap A$.
	Since $W\neq\varnothing$,
	the set
	$G\cap\overline A$
	is a nonempty relatively weakly open subset of $\overline A$.
	Hence, there exists $n\in\mathbb N$ satisfying $U_n\subset G\cap\overline A$.
	Consequently,
	\[
	V_n
	=
	U_n\cap A
	\subset
	G\cap A
	=
	W.
	\]
	Therefore, $(A,w)$ has a countable $\pi$-base.
\end{proof}

%

Theorem~\ref{thm:pi-base} actually reveals two rather different mechanisms behind the existence of countable weak $\pi$-bases.
The shrinking case is governed by the approximation of functionals by finite-dimensional projections, whereas the boundedly complete case is based on duality and the Radon--Nikodým property.

Both arguments admit independent generalizations.
We record them below as permanence results.
The first extends the shrinking argument to general Schauder decompositions, while the second applies to unconditional sums over boundedly complete sequence spaces.

%

\begin{proposition}\label{prop:shrinking-decomposition}
	Let
		$$
		X:=\overline{\bigoplus_{n=1}^{\infty}X_n}
		$$
	be a Schauder decomposition, and let $P_m:X\longrightarrow Y_m:=X_1\oplus\cdots\oplus X_m$ denote the corresponding partial-sum projections. 
	Assume that $\|f\circ (\operatorname{Id}-P_m)\|\longrightarrow 0$ for every $f\in X^*$, and that every bounded convex subset of $Y_m$ has a countable $\pi$-base for its relative weak topology, for every $m\in\mathbb N$. 
	Then, every bounded convex subset of $X$ has a countable $\pi$-base for its relative weak topology.
\end{proposition}

\begin{proof}
	The argument follows the scheme of \cite[Example~6.6]{AvilesKadetsMartinMeriShepelska2010}. 
	Let $A\subset X$ be a nonempty bounded convex set. 
	For every $m\in\mathbb N$, the set $P_m(A)$ is a bounded convex subset of $Y_m$. 
	Hence, there exists a countable $\pi$-base $\mathcal U_m := \{U_{m,k} \colon k\in\mathbb N\}$ for the relative weak topology of $P_m(A)$. 
	For every $m,k\in\mathbb N$, define $\widetilde U_{m,k} := P_m^{-1}(U_{m,k})\cap A$.
	Since $P_m$ is weak-to-weak continuous, each $\widetilde U_{m,k}$ is a nonempty relatively weakly open subset of $A$. 
	We claim that $\mathcal U := \{\widetilde U_{m,k} \colon m,k\in\mathbb N\}$ is a $\pi$-base for the relative weak topology of $A$.
	
	Let $W$ be a nonempty relatively weakly open subset of $A$. 
	Choose $a\in W$ and weak neighborhoods $U$ and $V$ of $0$ in $X$ such that $V+V\subset U$ and $(a+U)\cap A\subset W$.
	Since $P_ma\to a$ in norm, there exists $m_0\in \mathbb N$ such that we have $P_ma\in a+V$ for all $m\geq m_0$. 
	Consequently, for all $m\geq m_0$, we have that $C_m:=(a+V)\cap P_m(A)$ is a nonempty relatively weakly open subset of $P_m(A)$. 
	Thus, for every $m\geq m_0$, there exists $k_m\in\mathbb N$ such that $U_{m,k_m}\subset C_m$.
	Assume, by means of contradiction, that $\widetilde U_{m,k_m}\not\subset W$ for every $m\geq m_0$. 
	Then, we may choose $x_m\in \widetilde U_{m,k_m}\setminus W$ for every $m\geq m_0$.
	Set $z_m:=(I-P_m)x_m$ for every $m\geq m_0$.
	Since $A$ is bounded, there exists $M>0$ such that $\|x_m\|\leq M$ for every $m\geq m_0$. 
	Therefore, for every $f\in X^*$, we have
		$$
		|f(z_m)| = |f\circ (\operatorname{Id}-P_m)(x_m)| \leq M\|f\circ (\operatorname{Id}-P_m)\|\to 0.
		$$
	Hence, $z_m\to 0$ weakly.
	In particular, there exists $m_1\geq m_0$ such that $z_m\in V$ for all $m\geq m_1$.
	On the other hand, $P_mx_m\in U_{m,k_m}\subset a+V$ for every $m\geq m_0$.
	It follows that, for all $m\geq m_1$,
		\[
		x_m=P_mx_m+z_m\in a+V+V\subset a+U.
		\]
	Since $x_m\in A$ for all $m\geq m_0$, we obtain $x_m\in(a+U)\cap A\subset W$ for every $m\geq m_1$, contradicting the choice of $x_m$. 
	Thus, some member of $\mathcal U$ is contained in $W$, and $\mathcal U$ is a countable $\pi$-base for $A$.
\end{proof}

\begin{proposition}\label{prop:boundedly-complete-sum}
	Let $E$ be a Banach sequence space whose canonical basis is $1$-unconditional and boundedly complete, and let $(X_n)_{n=1}^{\infty}$ be a sequence of Banach spaces such that $X_n^*$ is separable for every $n\in\mathbb N$. 
	Then, every bounded convex subset of
		$$
		X:=\left[\bigoplus_{n=1}^{\infty}X_n\right]_E
		$$
	has a countable $\pi$-base for its relative weak topology.
\end{proposition}

\begin{proof}
	Since $X_n^*$ is separable, $X_n$ is separable for every $n\in\mathbb N$. Hence, $X$ is separable. 
	By \cite[Lemma~6.9]{AvilesKadetsMartinMeriShepelska2010}, it is enough to prove that, for every nonempty closed bounded convex subset $A\subset X$ and $\varepsilon>0$, there exist $x_0\in A$ and a sequence $(U_k)_{k=1}^{\infty}$ of relatively weakly open subsets of $A$ such that, for every weak neighborhood $V$ of $x_0$ in $X$, we have $U_k\subset V+\varepsilon B_X$ for some $k\in\mathbb N$.
	
	Fix a nonempty closed bounded convex set $A\subset X$ and $\varepsilon>0$. We first obtain a relatively weakly open subset of $A$ on which the tails are uniformly small, following the argument of \cite[Theorem~3.12]{AvilesKadetsMartinMeriShepelska2010}.
	
	Consider the bounded subset
		$$
		A_E:= \left\{ a=(a_n)_{n=1}^{\infty}\in E \colon \text{there exists } (x_n)_{n=1}^{\infty}\in A \text{ such that }|a_n|=\|x_n\|\text{ for every }n\in \mathbb N\right\},
		$$
	and set $K := \overline{\operatorname{conv}}(A_E)$. 
	Since the canonical basis of $E$ is boundedly complete, $E$ is isomorphic to a separable dual space and therefore has the Radon--Nikodým property (recall \cite[Theorem~3.2.15]{AlbiacKalton} and \cite[Proposition~5.5.6]{AlbiacKalton}). 
	Thus, $K$ is dentable (see \cite[Section~2]{Bourgin1983}).
	Hence, there are $b = (b_n)_{n=1}^\infty \in E^*$ and $\alpha\in\mathbb R$
	such that $S(K,b,\alpha) = \{a\in K \colon \operatorname{Re}b(a)>\alpha\} \neq \varnothing$ and $\operatorname{diam}(S(K,b,\alpha))<\varepsilon/4$. 
	Note that
		\begin{equation}\label{eq:sup}
			\sup_{a\in K} \operatorname{Re}b(a)=\sup_{a\in A_E}\operatorname{Re}b(a),
		\end{equation}
	Indeed, since $A_E\subset K$, one inequality is immediate. 
	For the reverse one, let $M:=\sup_{a\in A_E}\operatorname{Re}b(a)$.
	If $u\in\operatorname{conv}(A_E)$, then there are $a_1,\ldots,a_N\in A_E$ and $\lambda_1,\ldots,\lambda_N\geq 0$ with $\sum_{j=1}^N\lambda_j=1$ such that $u=\sum_{j=1}^N\lambda_j a_j$.
	Therefore,
		$$
		\operatorname{Re}b(u) = \sum_{j=1}^N\lambda_j\operatorname{Re}b(a_j) \leq \sum_{j=1}^N\lambda_j M = M.
		$$
	Thus, $\sup_{u\in\operatorname{conv}(A_E)}\operatorname{Re}b(u)\leq M$.
	Since $\operatorname{Re}b$ is continuous, it follows that $\sup_{a\in K}\operatorname{Re}b(a)\leq M$.
	As \eqref{eq:sup} holds, we may choose the slice $S(K,b,\alpha)$ so that it intersects $A_E$. 
	Therefore, $R_E:=\left\{ a\in A_E \colon \operatorname{Re} b(a)>\alpha\right\} \neq \varnothing$ and $\operatorname{diam}(R_E) < \varepsilon/4$. 
	Since $A_E$, and hence $K$, is invariant under coordinatewise changes of signs (or phases in the complex case), we may compose $b$ with a suitable coordinatewise isometry of $E$ and assume that $b_n\geq0$ for every $n\in\mathbb N$.
	Choose $x=(x_n)_{n=1}^\infty \in A$ such that $u:=(\|x_n\|)_{n=1}^{\infty}\in R_E$.
	By Hahn--Banach theorem, for every $n\in\mathbb N$, take $x_n^*\in S_{X_n^*}$ satisfying $x_n^*(x_n)=\|x_n\|$.
	Define $f\in X^*$ by
		\[
		f((y_n)_{n=1}^{\infty}) := \sum_{n=1}^{\infty}b_nx_n^*(y_n).
		\]
	Consider the relatively weakly open subset $S := \left\{ y\in A \colon  \operatorname{Re}f(y)>\alpha\right\}$.
	Notice that $x\in S$. Moreover, if $y=(y_n)_{n=1}^{\infty}\in S$, then
		$$
		\sum_{n=1}^{\infty}b_n\|y_n\| \geq \operatorname{Re}\sum_{n=1}^{\infty}b_nx_n^*(y_n) > \alpha.
		$$
	Thus, $v_y := (\|y_n\|)_{n=1}^{\infty}\in R_E$.
	Since $\operatorname{diam}(R_E)<\varepsilon/4$, we obtain $\|v_y-u\|_E<\varepsilon/4$ for every $y\in S$.
	Choose $m\in\mathbb N$ such that
		$$
		\left\|(0,\ldots,0,\|x_{m+1}\|,\|x_{m+2}\|,\ldots)\right\|_E < \frac{\varepsilon}{4}.
		$$
	Then, for every $y=(y_n)_{n=1}^\infty \in S$, we have
		\begin{align}
			\|y-P_m y\| & = \left\|(0,\ldots,0,\|y_{m+1}\|,\|y_{m+2}\|,\ldots)\right\|_E \nonumber\\
			& \leq \left\|(0,\ldots,0,\|x_{m+1}\|,\|x_{m+2}\|,\ldots)\right\|_E + \|v_y-u\|_E \nonumber \\
			& < \frac{\varepsilon}{2}\label{boundS}.
		\end{align}
	Fix $x_0\in S$. 
	Since $Y_m^*$ is separable, the weak topology on bounded subsets of $Y_m$ is metrizable. 
	Hence, there exists a countable local base $\{V_k \colon k\in\mathbb N\}$ at $P_mx_0$ for the relative weak topology of $P_m(S)$. 
	For every $k\in\mathbb N$, define $U_k:=P_m^{-1}(V_k)\cap S$.
	Each $U_k$ is a relatively weakly open subset of $A$ containing $x_0$.
	Let $V$ be an arbitrary weak neighborhood of $x_0$ in $X$. 
	Since $\|x_0-P_mx_0\|<\varepsilon/2$ by \eqref{boundS}, the set $\left(V+\frac{\varepsilon}{2}B_X\right)\cap P_m(S)$ is a relative weak neighborhood of $P_mx_0$ in $P_m(S)$. Therefore, there exists $k\in\mathbb N$ such that $V_k\subset V+\frac{\varepsilon}{2}B_X$.
	If $y\in U_k$, then $P_my\in V_k\subset V+\frac{\varepsilon}{2}B_X$ and $\|y-P_my\|<\varepsilon/2$ by \eqref{boundS}.
	Consequently, $y\in V+\varepsilon B_X$.
	Indeed, simply note that
		$$
		y = P_my+(y-P_my) = v+z+(y-P_my),
		$$
	where $v\in V$ and $z\in \frac{\varepsilon}{2}B_X$, and also that
		$$
		\|z+(y-P_my)\| \leq \|z\|+\|y-P_my\| < \frac{\varepsilon}{2}+\frac{\varepsilon}{2}=\varepsilon.
		$$	
	Thus, $U_k\subset V+\varepsilon B_X$.
	The conclusion now follows from \cite[Lemma~6.9]{AvilesKadetsMartinMeriShepelska2010}.
\end{proof}



\section{$k$-unconditional bases with non-SCD unit balls}\label{sec:non-scd}

Throughout this final section, we work over the real scalar field. We now turn to the quantitative unconditional renorming problem posed in~\cite[Question~5.4]{LooPerreau2026}.
We shall prove that, for every $k>1$, there exists a Banach space with a $k$-unconditional basis whose unit ball fails to be SCD.

We shall use the binary tree space considered by L\~oo and Perreau in~\cite{LooPerreau2026}, which goes back to Talagrand's construction of Banach spaces generated by adequate families~\cite{Talagrand1979,Talagrand1984}.
For completeness, we recall its definition.

Let
	\[
	T := \{\varnothing\}\cup \bigcup_{n=1}^{\infty} \{0,1\}^n
	\]
be the infinite binary tree.
If $s,t\in T$, we write $s\preceq t$ whenever $s$ is an initial segment of $t$.
A \emph{branch} of $T$ is a maximal chain for this order.
Let $\mathcal A$ be the family of branches of $T$.
We define a norm on $c_{00}(T)$ by
	\[
	\|x\| := \sup_{\beta\in \mathcal A} \sum_{t\in \beta} |x(t)|,
	\]
for every $x\in c_{00}(T)$.
The binary tree space $X_T$ is the completion of $c_{00}(T)$ for this norm.
For each $t\in T$, we denote by $e_t$ the canonical unit vector at $t$.
Then $(e_t)_{t\in T}$ is a normalized $1$-unconditional basis of $X_T$.

We shall also use the following elementary facts.
If $\beta \in \mathcal A$, then the functional
	\[
	f_\beta(x) := \sum_{t\in\beta} x(t)
	\]
is well-defined and belongs to $S_{X_T^*}$.
Moreover, if $x\in X_T^+$ (the positive cone of $X_T$), then
	\[
	\|x\| = \sup_{\beta\in \mathcal A}\sum_{t\in\beta} x(t).
	\]
This supremum is attained.
In particular, if $x\in S_{X_T}^+$, then there exists $\beta\in \mathcal A$ such that
	\[
	f_\beta(x) = \sum_{t\in\beta} x(t)=1.
	\]

From now on, set $X:=X_T$.
Denote by \(X^+\) and \(X^-\) the positive and negative cones of \(X\), respectively, and set
	\[
	B_X^+ := B_X\cap X^+ \qquad \text{and} \qquad B_X^- := B_X\cap X^-,
	\]
and
	\[
	S_X^+ := S_X\cap X^+.
	\]
For every \(\tau\in(0,1)\), define
	\[
	C_\tau := \overline{\operatorname{conv}} \left(B_X^+ \cup B_X^- \cup \tau B_X\right).
	\]
Let \(\|\cdot\|_\tau\) denote the Minkowski functional associated with \(C_\tau\). 
Since $\tau B_X \subset C_\tau \subset B_X$, the set \(C_\tau\) is closed, convex, balanced and absorbing.
Therefore, \(\|\cdot\|_\tau\) defines an equivalent norm on \(X\), and its closed unit ball coincides with \(C_\tau\).
Moreover, the inclusions above imply that
	\begin{equation}\label{eq:norm-equivalence-tau}
		\|x\| \le \|x\|_\tau \le \tau^{-1}\|x\|,
	\end{equation}
for all $x\in X$.

\medskip

We first establish the following analogue of \cite[Proposition~3.10]{LooPerreau2026}.
	
\begin{proposition}\label{prop:analogue-3.10}
	If \(x\in S_X^+\) is a SCD point of \(C_\tau\), then \(x\) is a SCD point of \(B_X^+\).
\end{proposition}
	
\begin{proof}
	Let \((S_n)_{n=1}^\infty\) be a sequence of slices of \(C_\tau\) which is determining for \(x\) in \(C_\tau\).
	Define
		\[
		M := \{ n\in\mathbb N \colon S_n\cap B_X^+\neq\varnothing \} \qquad \text{and} \qquad D_\tau := \overline{\operatorname{conv}} \left(B_X^-\cup \tau B_X\right).
		\]
	Then,
		\begin{equation}\label{eq:Ctau-decomposition}
			C_\tau = \overline{\operatorname{conv}} \left(B_X^+\cup D_\tau\right).
		\end{equation}
	We shall prove that the sequence $\left(S_n\cap B_X^+\right)_{n\in M}$ is determining for \(x\) in \(B_X^+\).
	
	We first claim that $S_n\cap D_\tau\neq\varnothing$ for all $n\notin M$.
	Fix \(n\notin M\). 
	Since \(S_n\) is a slice of \(C_\tau\), there exist \(f_n\in X^*\setminus\{0\}\) and \(\alpha_n>0\) such that $S_n = S(C_\tau,f_n,\alpha_n)$.
	As \(S_n\cap B_X^+ = \varnothing\) and by \eqref{eq:Ctau-decomposition}, we have
		\[
		\sup_{z\in B_X^+} f_n(z) \le \sup_{z\in C_\tau}f_n(z)-\alpha_n < \sup_{z\in C_\tau}f_n(z) = \max \left\{\sup_{z\in B_X^+}f_n(z),\,\sup_{z\in D_\tau}f_n(z)\right\}.
		\]
	Thus, it follows that
		\[
		\sup_{z\in D_\tau}f_n(z) = \sup_{z\in C_\tau}f_n(z).
		\]
	Consequently, $S_n\cap D_\tau\neq\varnothing$, as claimed.
	
	We next show that $M\neq\varnothing$. 
	Assume otherwise, then $S_n\cap B_X^+=\varnothing$ for every $n\in\mathbb N$. 
	By the argument in the previous paragraph, this implies $S_n\cap D_\tau\neq\varnothing$ for every $n\in \mathbb N$. 
	Choose $x_n\in S_n\cap D_\tau$ for each $n\in \mathbb N$. 
	Since $(S_n)_{n=1}^\infty$ is determining for $x$ in $C_\tau$, we obtain $x\in \overline{\operatorname{conv}}\{x_n:n\in\mathbb N\}\subset D_\tau$.
	However, since $x\in S_X^+$, there exists a branch $\beta$ such that $f_\beta(x)=1$, while $f_\beta(z)\leq \tau$ for every $z\in D_\tau$. 
	Indeed, $f_\beta(z)\leq 0$ for $z\in B_X^-$, while $f_\beta(z)\leq \tau$ for $z\in \tau B_X$, and the inequality extends to $D_\tau$ by convexity and continuity. 
	Since $\tau<1$, this contradicts $x\in D_\tau$, and thus $M\neq\varnothing$.
	
	Now we claim that, for every $n\in M$, the set $S_n\cap B_X^+$ is a slice of $B_X^+$. 
	Indeed, if $S_n=S(C_\tau,f_n,\alpha_n)$, then $S_n\cap B_X^+\neq\varnothing$ implies
		$$
		\alpha_n^+ := \sup_{z\in B_X^+} f_n(z) - \left(\sup_{z\in C_\tau} f_n(z)-\alpha_n\right)>0,
		$$
	and $S_n\cap B_X^+ = S(B_X^+,f_n,\alpha_n^+)$.
	
	For every $n\in M$, choose $x_n^+\in S_n\cap B_X^+$; and for every $n\notin M$, choose $x_n^-\in S_n\cap D_\tau$.
	Define
		\[
		x_n := 
		\begin{cases}
			x_n^+ & \text{if } n\in M, \\
			x_n^- & \text{if } n\notin M.
		\end{cases}
		\]
	Since \((S_n)_{n=1}^\infty\) is determining for \(x\) in \(C_\tau\), we have $x\in \overline{\operatorname{conv}} \{x_n:n\in\mathbb N\}$.
	Fix $\varepsilon>0$. 
	Since $x\in \overline{\operatorname{conv}}\{x_n \colon n\in\mathbb N\}$, there is a finite convex combination of the vectors $x_n$ whose distance to $x$ is less than $\varepsilon$. 
	Separating the terms with indices in $M$ from those with indices outside $M$, we obtain
		\begin{equation}\label{eq:approximation-x}
			\|x-(\lambda y+(1-\lambda)z)\| < \varepsilon,
		\end{equation}
	where $\lambda\in[0,1]$, $y\in\operatorname{conv}\{x_n^+ \colon n\in M\}$, and, if
	$\lambda<1$, then $z\in\operatorname{conv}\{x_n^- \colon n\notin M\}\subset D_\tau$.
	If $\lambda=1$, the choice of $z$ is irrelevant.
	Since \(x\in S_X^+\), by the definition of the norm of the binary tree space there is a branch \(\beta\) such that
		\[
		\sum_{t\in\beta}x(t)=\|x\|=1.
		\]
	Define
		\[
		f_\beta(w) := \sum_{t\in\beta} w(t),
		\]
	for all $w\in X$.
	Again, by the definition of the norm of the binary tree space, $\|f_\beta\|=1$ and $f_\beta(x)=1$.
	Since \(y\in B_X^+\), we have $f_\beta(y)\le1$.
	We next show that
		\begin{equation}\label{eq:fbeta-Dtau}
			f_\beta(w)\le\tau,
		\end{equation}
	for all $w\in D_\tau$.
	Indeed, if \(w\in B_X^-\), then $f_\beta(w)\le0$.
	On the other hand, if \(w\in\tau B_X\), then $f_\beta(w) \le \tau\|f_\beta\| = \tau$.
	Since \(D_\tau\) is the closed convex hull of \(B_X^-\cup\tau B_X\), inequality \eqref{eq:fbeta-Dtau} follows by convexity and continuity.
	Applying \(f_\beta\) to \eqref{eq:approximation-x}, we obtain
		\[
		1-\varepsilon < f_\beta(\lambda y+(1-\lambda)z) = \lambda f_\beta(y) + (1-\lambda)f_\beta(z) \le \lambda+(1-\lambda)\tau = 1-(1-\lambda)(1-\tau),
		\]
	which yields
		\begin{equation}\label{eq:lambda-estimate}
			1-\lambda < \frac{\varepsilon}{1-\tau}.
		\end{equation}
	Finally, by the triangle inequality, \eqref{eq:approximation-x}, \eqref{eq:lambda-estimate} and the fact that \(y,z\in B_X\), we have
		\begin{align*}
			\|x-y\| & \le \|x-(\lambda y+(1-\lambda)z)\| + \|(\lambda y+(1-\lambda)z)-y\| \\
			& < \varepsilon + \|(1-\lambda)(z-y)\| \\
			& \le \varepsilon + 2(1-\lambda) \\
			& < \varepsilon + \frac{2\varepsilon}{1-\tau}.
		\end{align*}
	Since \(\varepsilon>0\) is arbitrary, we have $x \in \overline{\operatorname{conv}} \{x_n^+ \colon n\in M\}$.
	Therefore, the sequence $\left(S_n\cap B_{X^+}\right)_{n\in M}$ is determining for \(x\) in \(B_X^+\), proving that \(x\) is a SCD point of \(B_X^+\).
\end{proof}

\begin{proposition}
	\label{prop:tau-unconditional-non-scd}
	For every \(\tau\in(0,1)\), the canonical basis
	\((e_t)_{t\in T}\) is \(\tau^{-1}\)-unconditional for the norm
	\(\|\cdot\|_\tau\), and the unit ball \(C_\tau\) is not SCD.
\end{proposition}

\begin{proof}
	We first prove the unconditionality estimate. 
	Let $\varepsilon=(\varepsilon_t)_{t\in T}\in\{-1,1\}^T$, and define the sign-change operator $S_\varepsilon:X\longrightarrow X$ by
		\[
		S_\varepsilon\left(\sum_{t\in T} a_t e_t\right) := \sum_{t\in T}\varepsilon_t a_t e_t.
		\]
	Since the canonical basis of \(X\) is \(1\)-unconditional for the original norm of \(X\), we have $\|S_\varepsilon x\| = \|x\|$ for all $x\in X$.	
	Using \eqref{eq:norm-equivalence-tau}, we obtain
		\[
		\|S_\varepsilon x\|_\tau \le \tau^{-1}\|S_\varepsilon x\| = \tau^{-1}\|x\| \le \tau^{-1}\|x\|_\tau.
		\]
	Therefore, $\|S_\varepsilon\|_{(X,\|\cdot\|_\tau)} \le \tau^{-1}$, showing that the canonical basis is \(\tau^{-1}\)-unconditional for \(\|\cdot\|_\tau\).
	
	\medskip
	
	We now prove that \(C_\tau\) is not SCD.
	By \cite[Proposition~3.9]{LooPerreau2026}, there exists a point $x\in S_X^+$ which is not a SCD point of $B_X^+$. 
	Assume, by means of contradiction, that $C_\tau$ were SCD. 
	Then, every point of $C_\tau$ would be a SCD point of $C_\tau$. 
	In particular, $x$ would be a SCD point of $C_\tau$.
	Applying Proposition~\ref{prop:analogue-3.10}, it would follow that $x$ is a SCD point of $B_X^+$, contradicting the choice of $x$. 
	Hence, $C_\tau$ is not SCD.
\end{proof}

As an immediate consequence, we obtain the following positive answer
to \cite[Question~5.4]{LooPerreau2026}.

\begin{theorem}
	\label{thm:k-unconditional-non-scd}
	For every \(k>1\), there exists a Banach space with a
	\(k\)-unconditional basis whose unit ball is not SCD.
\end{theorem}

\begin{proof}
	Fix \(k>1\), and choose $\tau := 1/k$.
	By Proposition~\ref{prop:tau-unconditional-non-scd}, the space $(X,\|\cdot\|_\tau)$ has a \(k\)-unconditional basis and its unit ball $C_\tau$ is not SCD.
\end{proof}

%


\end{document}